\documentclass[smallextended]{svjour3}

\usepackage{color}
\usepackage{float}
\usepackage{bm}
\usepackage{amssymb,amsmath}
\usepackage[noadjust]{cite}
\usepackage{graphicx}
\usepackage{comment}
\usepackage{marginnote}
\usepackage{paralist}

\smartqed

\usepackage{mathrsfs}
\usepackage{mathtools}
\usepackage{algorithmic}
\usepackage{enumitem}
\allowdisplaybreaks 

\makeatletter
\let\oldfootnote\footnote
\def\footnote{\@ifstar\footnote@star\footnote@nostar}
\def\footnote@star#1{{\let\thefootnote\relax\footnotetext{#1}}}
\def\footnote@nostar{\oldfootnote}
\makeatother
\usepackage[noadjust]{cite}
\usepackage{subfigure}

\floatstyle{ruled}
\newfloat{algorithm}{tbp}{loa}
\providecommand{\algorithmname}{Algorithm}
\floatname{algorithm}{\protect\algorithmname}
\newcounter{algo}

\renewcommand{\d}{\mathbf{d}}

\newcommand{\argmax}{\mathop{\rm argmax}}

\newcommand{\bx}{{\mathbf{x}}}
\newcommand{\by}{{\mathbf{y}}}

\usepackage[footnotesize]{caption}
\usepackage{flushend}

\allowdisplaybreaks[4]
\makeatother

\begin{document}

%%%%%%%%%

\title{Asynchronous Parallel  Algorithms for Nonconvex Big-Data Optimization\\ Part II: Complexity and Numerical Results}

\author{Loris Cannelli \and Francisco Facchinei  \and  Vyacheslav Kungurtsev \and Gesualdo Scutari}

\institute{ Loris Cannelli and Gesualdo Scutari, 
School of Industrial Engineering, Purdue University, USA \email{$<$lcannelli, gscutari$>$@purdue.edu}. \and
Francisco Facchinei, Department of Computer, Control, and Management Engineering Antonio Ruberti, University of Rome La Sapienza,  Roma, Italy  \email{francisco.facchinei@uniroma1.it}.\and Vyacheslav Kungurtsev
Dept. of Computer Science, Faculty of Electrical Engineering, Czech Technical University in Prague, Czech
 \email{vyacheslav.kungurtsev@fel.cvut.cz}.\newline  The work of Cannelli and Scutari was supported by the USA National Science Foundation (NSF) under Grants CIF 1564044, CCF 1632599 and CAREER Award No. 1555850, and the Office of Naval Research (ONR) Grant N00014-16-1-2244. 
 %Kungurtsev
%was supported by the Cisco-Czech Technical University Sponsored Research Agreement project WP5
%and the Czech Science Foundation project 17-26999S.}
}

\date{\today}
%\date{19th March 2014}

\maketitle
\begin{abstract}
 We present complexity and numerical results for a new asynchronous parallel algorithmic method
for the minimization of the sum of a
 smooth nonconvex function and a convex nonsmooth regularizer, subject to both convex and nonconvex constraints.
The proposed  method hinges on successive convex approximation techniques and a novel probabilistic model that  
captures key elements of  modern computational architectures  and asynchronous implementations in a more faithful way 
than  state-of-the-art models. In the companion paper \cite{companion2}  we provided a detailed description on the 
probabilistic model and gave  convergence
results for a diminishing stepsize version of our method. Here, we provide theoretical complexity results for a fixed stepsize version of the method and report  extensive numerical comparisons on both
convex and nonconvex problems demonstrating the efficiency of our  approach.


  \end{abstract}
  
  %%%%%%%%%%%%%

	\section{Introduction}
The rise of the big data challenge  has created a strong demand for highly parallelizable algorithms solving huge 
optimization problems quickly and reliably. 
Recent research on algorithms incorporating asynchrony
 has given promising theoretical and computational results but, as discussed thoroughly in the companion paper \cite{companion2}, to which we refer the reader for a comprehensive bibliographical review, there is still room for huge improvements. In \cite{companion2}
we presented a novel probabilistic model we believe gives a more accurate description of asynchrony as encountered in modern computational architectures and proposed a new diminishing stepsize minimization method for which we proved convergence to stationary points. Here,
we complete the analysis by proving complexity results for a fixed stepsize version of the method proposed in \cite{companion2} and by reporting numerical test
results comparing our algorithm to   state-of-the-art methods on  both convex and nonconvex problems.

	We consider the minimization of a smooth (possibly) \textit{nonconvex} function $f$ and a nonsmooth block-separable 
	convex one $G$ subject to convex constraints $\mathcal{X}$ and local nonconvex ones $c_{j_i}(\mathbf{x}_i)\leq0$. The formulation reads
	\begin{equation}
	\begin{aligned}
	& \underset{\mathbf{x}=(\mathbf{x}_1,\ldots,\mathbf{x}_N)}{\min} & & F(\mathbf{x})\triangleq f(\mathbf{x})+G(\mathbf{x}) \\
	& \text{subject to} & & \begin{rcases}\mathbf{x}_i\in\mathcal{X}_i,\quad i=1,\ldots, N\\
	c_{j_1}(\mathbf{x}_1)\leq0,\quad j_1=1,\ldots,m_1\\
	\vdots\\
	c_{j_N}(\mathbf{x}_N)\leq0,\quad j_N=m_{N-1}+1,\ldots,m_N.
	\end{rcases}\triangleq\mathcal{K}
	\end{aligned}
	\label{ncc_1}
	\end{equation}
	We denote by  $\mathcal{X}=\mathcal{X}_1\times\ldots\times\mathcal{X}_N$ the  Cartesian product of the lower dimensional 
	closed, convex sets $\mathcal{X}_i\subseteq\mathbb{R}^{n_i}$.% $\mathbf x\in \Re^n$ is partitioned %accordingly: $
	%\mathbf x = (\mathbf  x_1, \ldots, \mathbf x_N)$, with each $\mathbf  x_i \in \Re^{n_i}$. 
	We consider the presence of 
	$m_i-m_{i-1}$ nonconvex local constraints $c_{j_i}:\mathcal{X}_i\rightarrow\mathbb{R}$ with $j_i=m_{i-1},\ldots,m_i$, 
	for each block of variables $\mathbf{x}_i$, with $m_i\geq0$, for $i\in\mathcal{N}\triangleq \{1,\ldots,N\}$ and $m_0=0$.  We denote  by  $\mathcal{K}_i$ the set 
	$
	\mathcal{K}_i\triangleq \left\{\mathbf{x}_i\in\mathcal{X}_i:\,
	c_{j_i}(\mathbf{x}_i)\leq 0,\quad j_i = m_{i-1}+1,\ldots,m_i\right\}$.
	The function $f$ is  smooth (not necessarily convex or separable)
	and $G$  is convex, separable, and  possibly nondifferentiable.\\
	\noindent \textbf{Assumption A} We  make the following blanket assumptions.
	\begin{description}[topsep=-2.0pt,itemsep=-2.0pt]
		\item[(A1)]  Each $\mathcal{X}_i$ is nonempty, closed and convex;\smallskip
		\item[(A2)] $f$ is $C^1$ on an open set containing $\mathcal{X}$;\smallskip
		\item[(A3)]  $\nabla_{\mathbf{x}_i} f$ is   Lipschitz continuous
		on $\mathcal{X}_i$ with a Lipschitz constant $L_f$ which is independent of $i$;\smallskip
		\item[(A4)] $G(\mathbf{x})\triangleq\sum_ig_i(\mathbf{x}_i)$, and each  $g_i(\mathbf{x}_i)$ is continuous, convex, 
and Lipschitz continuous with constant $L_g$ on $\mathcal{X}_i$ (but possibly nondifferentiable);
		\smallskip
		\item[(A5)] $\mathcal{K}$ is compact;\smallskip
		\item[(A6)]  Each $c_{j_i}:\mathcal{X}_i\rightarrow\mathbb{R}$ is continuously differentiable on $\mathcal{X}_i$, for 
		all $i\in\mathcal{N}$ and  $j_i\in\{m_{i-1}+1,\ldots,m_i\}$, $m_i\geq0$ and $m_0=0$.
	\end{description}\smallskip

\noindent
In this paper we study  the complexity  of a version of the algorithm proposed in the companion paper~\cite{companion2} that uses a fixed stepsize. We prove that the expected value of an appropriate stationarity measure  becomes smaller than $\epsilon$ after a number of 
iterations proportional to $1/\epsilon$. We also show that if the stepsize is small enough, then a linear speedup can be expected as the number of cores increases.
Although comparisons are difficult both  because our probabilistic  model is different from and  more accurate than most usually used in the literature, our complexity results seem comparable to the ones in the literature for both convex problems (see for example 
\cite{hong2014distributed,liu2015asynchronous,Mania_et_al_stochastic_asy2016,Hogwild!,peng2015arock}), and 
nonconvex problems (\cite{lian2015asynchronous} and
\cite{davis2016asynchronous,DavisEdmundsUdell}). 
In~\cite{lian2015asynchronous} complexity is analyzed for a stochastic gradient methods for {\em unconstrained,} smooth 
nonconvex problems. 
It is shown that  a number of iterations proportional to 
$1/\epsilon$ is needed in order to drive the expected value of the gradient below $\epsilon$. Similar results are proved 
\cite{davis2016asynchronous,DavisEdmundsUdell} for nonconvex, constrained problems.
However, recall that in \cite{companion2} it was observed that the probabilistic models used in
\cite{davis2016asynchronous,DavisEdmundsUdell,lian2015asynchronous} or, from another point of view, the implicit 
assumptions made in these papers, are problematic. Furthermore it should also be observed that the methods for 
nonconvex, constrained  problems in \cite{davis2016asynchronous,DavisEdmundsUdell} require the global solution of 
nonconvex subproblems, making these methods of uncertain practical use, in general.
Therefore,   the complexity results presented in this paper are of novel value and represent an advancement
in the state of the art of asynchronous algorithms for large-scale (nonconvex) optimization problems.

We also provide extensive numerical results for the diminishing stepsize method proposed in the companion 
paper 
\cite{companion2}. We compare the performance and the speedup of this method with the most advanced
state of the art algorithms for both convex and nonconvex problems. 
The results show that our method compares favorably to existing asynchronous methods. 

The paper is organized as follows. In the next section, after briefly summarizing some necessary elements from 
\cite{companion2}, we describe the algorithm and in Section 3 we give the complexity results. In Section 4 we report the 
numerical results and finally draw some conclusions in Section 5. All proofs are given in the Appendix.

\section{Algorithm}
In this section we  describe the asynchronous model and algorithm as proposed in the companion paper \cite[Sec.\,4]{companion2}. The difference with \cite{companion2}  is that here, in order to study the convergence rate of the algorithm,  we enforce a fixed  stepsize (rather than using a diminishing stepsize). % along with the associated computational model and the necessary assumptions. We refer the reader  to \cite{companion2} for a complete discussion and 
 %highlight the fact that the algorithm discussed in this paper is the same already presented in Section 4 of the Part I of our work \cite{companion2}, with the important difference that in order to study the convergence rate of our approach we enforce a fixed stepsize. 
 
	%We leverage the Successive Convex Approximation technique for solving problem \eqref{ncc_1}:  we follow an approach similar to the one proposed in \cite{scutari2014distributed} solving a sequence of strongly convex inner approximations of problem \eqref{ncc_1}. 
	
	We use a global
index $k$ to count iterations: whenever a core updates a block-component of
the current $\bx$, a new iteration $k \rightarrow k+1$ 
 is triggered (this iteration counter
is not required by the cores themselves to compute the updates).   %According to \cite[Sec.\,4]{companion2},
 Therefore, at each iteration $k$, there is a core that updates   in an independent and asynchronous fashion a block-component  $\bx_{i^k}$ of $\bx^k$ randomly
chosen,  %in
%the set $\mathcal N\triangleq \{1,\ldots, N\}$, 
thus generating the vector $\bx^{k+1}$.
Hence,
$\bx^{k}$ and $\bx^{k+1}$ only differ in the $i^k$-th   component.  	
	 To update    block  $i^k$, a core generally
does not have access to the global vector $\bx^k$, but will instead use   the possibly out-of-sync, delayed  coordinates
$\tilde \bx^k=\bx^{k-\d^k}\triangleq (\bx_1^{k-d^k_1},\ldots ,\bx_N^{k-d^k_N})$,   where  
 $d^k_i$ are some nonnegative integer numbers. Given $\tilde \bx^k$, to update the $i^k$-th component,    the following strongly convex problems is first solved  
 \begin{equation}
	\hat{\textbf{x}}_i^k(\tilde{\mathbf{x}}^{k}) = \underset{\textbf{x}_{i^k} \in \mathcal{K}_{i^k}(\tilde{\mathbf{x}}^{k}_{i^k})}{\arg\min} {\tilde{F}_{i^k}(\textbf{x}_{i^k};\tilde{\mathbf{x}}^{k})}\triangleq \tilde f_{i^k}(\bx_{i^k}; \tilde{\mathbf{x}}^{k}) + g_{i^k}(\bx_{i^k}),
	\label{ncc_2}
	\end{equation}
	where $\tilde f_{i}: \mathcal{X}_i\times \mathcal{K}\to \mathbb{R}$ and $\mathcal{K}_i(\bullet)$, with $i=1,\ldots, N$, represent  convex approximations of $f$ and $\mathcal{K}_i$, respectively, defined according to the rules listed below.  Then block  $i^k$ is updated according to 
\begin{equation} \mathbf{x}_{i^k}^{k+1}=\mathbf{x}^k_{i^k}+\gamma(\hat{\mathbf{x}}_{i^k}(\tilde{\mathbf{x}}^{k})-\mathbf{x}_{i^k}^k);\end{equation}
where $\gamma$ is a positive constant.

	%After having solved problem \eqref{ncc_2}, $\mathbf{x}_i^k$ is updated taking a step along the direction 
%$\hat{\mathbf{x}}_i(\mathbf{x}^k)-\mathbf{x}_i^k$ with a suitable steplength $\gamma>0$. 
%......

%To derive our complexity results,
%we use a fixed stepsize that must satisfy a certain bound. In practice, we use a diminishing stepsize rule, which both
%in most settings and in particular here, tends to result in faster and more reliable convergence.	

	 	We require the following  assumptions on the surrogate function $\tilde f_{i}$.\smallskip\\
	\noindent \textbf{Assumption B (On the surrogate functions).} Each $\tilde{f}_i$ is  a function continuously differentiable with respect to the first argument such that:
	\begin{description}%[topsep=-2.0pt,itemsep=-2.0pt]
		\item[ (B1)] $\tilde f_{i} (\mathbf{\bullet}; {{\by}})$ is uniformly strongly
		convex	on $ \mathcal{X}_i$ for all ${{\by}}\in \mathcal{K}$ with a strong convexity constant $c_{\tilde{f}}$ which is independent of  $i$ and $k$;\smallskip
		\item[  (B2)]   $\nabla \tilde f_{i} ({\by}_i;\by) = \nabla_{{\by}_i} f({\by})$, for all ${\by} \in  \mathcal{K}$;\smallskip
		\item[  (B3)]  $\nabla \tilde f_{i} (\mathbf{y}_i;\mathbf{\bullet})$
		is Lipschitz continuous
		on $ \mathcal{K}$,
		for all $\mathbf{y}_i \in  \mathcal{X}_i$, with a Lipschitz constant $L_B$ which is independent of  $i$ and $k$;\smallskip
		\item [  (B4)]$\nabla \tilde f_{i} (\mathbf{\bullet};\mathbf{y})$
		is Lipschitz continuous
		on $ \mathcal{X}_i$, for all $\mathbf{y} \in  \mathcal{K}$, with a Lipschitz constant $L_E$ which is independent of $i$ and $k$.
	\end{description}\smallskip
	$\mathcal{K}_i(\mathbf{y}_i)$ is a convex approximation of $\mathcal{K}_i$ defined by	\begin{equation}
	\mathcal{K}_i(\mathbf{y}_i)\triangleq\begin{cases}
	\tilde{c}_{j_i}(\mathbf{x}_i;\mathbf{y}_i)\leq 0,\quad j_i=m_{i-1}+1,\ldots,m_i\\
	\mathbf{x}_i\in\mathcal{X}_i
	\end{cases},\,i=1,\ldots,N,
	\label{ncc_3}
	\end{equation}
	where $\tilde{c}_{j_i}:\mathcal{X}_i\times\mathcal{K}_i\rightarrow\mathbb{R}$ is required to satisfy the following assumptions.\smallskip\\
	\noindent \textbf{Assumption C (On $\tilde{c}_{j_i}$'s).} For each $i\in\mathcal{N}$ and $j_i\in\{m_{i-1}+1,\ldots,m_i\}$, it holds: 
	\begin{description}[topsep=-2.0pt,itemsep=-2.0pt]
		\item[  (C1)]$\tilde{c}_{j_i}(\bullet;\mathbf{y})$ is convex on $\mathcal{X}_i$ for all $\mathbf{y}\in\mathcal{K}_i$;\smallskip
		\item[  (C2)] $\tilde{c}_{j_i}(\mathbf{y};\mathbf{y})=c_{j_i}(\mathbf{y})$, for all $\mathbf{y}\in\mathcal{K}_i$;\smallskip
		\item[  (C3)] $c_{j_i}(\mathbf{z})\leq\tilde{c}_{j_i}(\mathbf{z};\mathbf{y})$ for all $\mathbf{z}\in\mathcal{X}_i$ and $\mathbf{y}\in\mathcal{K}_i$;\smallskip
		\item[  (C4)] $\tilde{c}_{j_i}(\bullet;\bullet)$ is continuous on $\mathcal{X}_i\times\mathcal{K}_i$;\smallskip
		\item[  (C5)] $\nabla c_{j_i}(\mathbf{y})=\nabla_1\tilde{c}_{j_i}(\mathbf{y};\mathbf{y})$, for all $\mathbf{y}\in\mathcal{K}_i$;\smallskip
		\item[  (C6)] $\nabla\tilde{c}_{j_i}(\bullet;\bullet)$ is continuous on $\mathcal{X}_i\times\mathcal{K}_i$;\smallskip
		\item[ (C7)] Each $\tilde{c}_{j_i}(\bullet;\bullet)$ is Lipschitz continuous on $\mathcal{X}_i\times\mathcal{K}_i$;\medskip
	\end{description}
	where $\nabla \tilde{c}_{j_i}$ is
	the partial gradient of $\tilde{c}_{j_i}$ with respect to the first argument.

%Whenever a core updates a block-component of the current $\bx$, a new
%iteration $k \rightarrow k+1$ is triggered (this iteration counter is not required by the cores themselves to compute the
%updates). At iteration $k$, there is a core that updates a block-component  $\bx_{i^k}$ of $\bx^k$ randomly
%chosen  in
%the set $\mathcal N\triangleq \{1,\ldots, N\}$ thus generating the vector $\bx^{k+1}$.
%Therefore,
%$\bx^{k}$ and $\bx^{k+1}$ only differ in the $i^k$-th updated component. To update    block  $i^k$ a core generally
%does not have access to the global state $\bx^k$, but will instead use   the possibly out-of-sync, delayed  coordinates
%$\tilde \bx^k=\bx^{k-\d^k}\triangleq (\bx_1^{k-d^k_1},\ldots ,\bx_N^{k-d^k_N})$,   where
% $d^k_i$'s are some nonnegative integer numbers. 
%Thus problem~\eqref{ncc_2} is solved, but instead of evaluating the solution at $\bx^k$, each core computes
%$\hat{\mathbf{x}}_i(\tilde\bx^k)$.
%	After having solved problem \eqref{ncc_2}, $\mathbf{x}_i^k$ is updated taking a step along the direction 
%$\hat{\mathbf{x}}_i(\mathbf{x}^k)-\mathbf{x}_i^k$ with a suitable steplength $\gamma>0$. To derive our complexity results,
%we use a fixed stepsize that must satisfy a certain bound. In practice, we use a diminishing stepsize rule, which both
%in most settings and in particular here, tends to result in faster and more reliable convergence.	

	The randomness associated with the block selection procedure and the delayed information being used to compute the solution
to the subproblems is described in our model with the index-delay pair $(i,\mathbf{d})$ used at each iteration $k$ to update $\mathbf{x}^k$ being a realization of a random vector $\underline{\boldsymbol{\omega}}^k\triangleq (\underline{i}^k,\underline{\mathbf{d}}^k)$, taking values on $\mathcal N \times \mathcal D$ with some probability $p_{{i},{\mathbf{d}}}^k\triangleq \mathbb{P}({(\underline{i}^k,\underline{\mathbf{d}}^k)=({i},{\mathbf{d}}}))$, where ${\cal D}$ is the set    of all possible delay vectors.   Since each delay $d_i \leq \delta$ (see Assumption C below),  ${\cal D}$ is the set of all possible  $N$-length vectors whose components are integers between $0$ and $\delta$. More formally, let    $\Omega$ be the sample space  of   all  the  sequences $\{ (i^k, \mathbf{d}^k)\}_{k\in\mathbb{N}_+}$, and  let us define  the discrete-time, discrete-value   stochastic process $\underline{\boldsymbol{\omega}}$, where   $\{\underline{\boldsymbol{\omega}}^k(\omega)\}_{k\in\mathbb{N}_+}$  is a sample path of the process. The  $k$-th entry  $\underline{\boldsymbol{\omega}}^k(\omega)$ of $\underline{\boldsymbol{\omega}}(\omega)-$the $k$-th element of the sequence $\omega-$is a realization of the random vector  $\underline{\boldsymbol{\omega}}^k= (\underline{i}^k,\underline{\mathbf{d}}^k):\Omega \mapsto \mathcal N\times \mathcal D$. %Indeed, given $\bx^0$,   any trajectory of index-delay pairs $\{ (i^k, \mathbf{d}^k)\}_{k\in \mathbb{N}_+}$ inducing the updates $\{\mathbf{x}^k\}_{k\in \mathbb{N}_+}$ is a sample function of $\underline{\mathbf{Y}}$.
	%the evolution of Algorithm 1 can be fully described by  the discrete time, discrete value random  stochastic process $\underline{\mathbf{Y}}=\{\underline{\mathbf{Y}}(\omega):\,\omega\in \Omega\}$, where $\Omega$ is the countably infinite set of all  the  sequences $\{ (i^k, \mathbf{d}^k)\}_{k\in \mathbb{N}_+}$,  and   $\underline{\mathbf{Y}}(\omega)=\{\underline{\mathbf{Y}}^k(\omega)\}_{k\in \mathbb{N}_+}$  is a sample function  of the process, whose $k$-th entry  $\underline{\mathbf{Y}}^k(\omega)$ is a realization of the random variable  $\underline{\mathbf{Y}}^k= (\underline{i}^k,\underline{\mathbf{d}}^k):\Omega \mapsto \mathcal N\times \mathcal D$. Indeed, given $\bx^0$,   any trajectory of index-delay pairs $\{ (i^k, \mathbf{d}^k)\}_{k\in \mathbb{N}_+}$ inducing the updates $\{\mathbf{x}^k\}_{k\in \mathbb{N}_+}$ is a sample function of $\underline{\mathbf{Y}}$.

	The stochastic process   $\underline{\boldsymbol{\omega}}$ is  fully defined  once the joint finite-dimensional probability mass functions $p_{\underline{\boldsymbol{\omega}}^{0:k}}(\boldsymbol{\omega}^{0:k})\triangleq \mathbb{P}(\underline{\boldsymbol{\omega}}^{0:k}=\boldsymbol{\omega}^{0:k})$  are given, for all admissible tuples $\boldsymbol{\omega}^{0:k}$ and $k$, where we used the shorthand notation $\underline{\boldsymbol{\omega}}^{0:t} \triangleq (\underline{\boldsymbol{\omega}}^0,\underline{\boldsymbol{\omega}}^1, \ldots,
	\underline{\boldsymbol{\omega}}^t)$ (the first $t+1$ random variables), and
	${\boldsymbol{\omega}}^{0:t} \triangleq ({\boldsymbol{\omega}}^0,{\boldsymbol{\omega}}^1, \ldots,{\boldsymbol{\omega}}^t)$ ($t+1$ possible values
	for the random variables  $\underline{\boldsymbol{\omega}}^{0:t}$). In fact, this joint distribution induces a valid probability space $(\Omega, \mathcal{F},P)$ over which $\underline{\boldsymbol{\omega}}$ is well-defined and has $p_{\underline{\boldsymbol{\omega}}^{0:k}}$ as its finite-dimensional distributions. 

This process fully describes the evolution of Algorithm 1. Indeed,  given a starting point $\bx^0$,  the trajectories  of   the  variables $\bx^k$ and  $ \bx^{k-\mathbf{d}}$ are  completely determined once a sample path $\{ (i^k, \mathbf{d}^k)\}_{k\in \mathbb{N}_+}$ is drawn by $\underline{\boldsymbol{\omega}}$.

Finally, we need to define the conditional probabilities  $p((i,\mathbf{d})\,|\, \boldsymbol{\omega}^{0:k})\triangleq
\mathbb{P}(\underline{\boldsymbol{\omega}}^{k+1}=(i,\mathbf{d})|
\underline{\boldsymbol{\omega}}^{0:k}={\boldsymbol{\omega}}^{0:k})$. We require a few minimal conditions on the probabilistic model, as stated next.\\
	\noindent\textbf{Assumption D} Given the global model described in  Algorithm 1 and the stochastic process
	$\underline{\boldsymbol{\omega}}$, suppose that
	\begin{description}%[topsep=-1.0pt,itemsep=-2.0pt]
		\item[(D1)] There   exists a $\delta\in \mathbb{N}_+$, such that $d^k_{i} \leq \delta$,  for all $i$ and $k$;
		
		\item[(D2)]  For all $i=1,\ldots, N$ and $\boldsymbol{\omega}^{0:k-1}$ such that
		$p_{\underline{\boldsymbol{\omega}}^{0:k-1}}(\boldsymbol{\omega}^{0:k-1})>0$, it holds $$\sum_{\mathbf{d}\in
			\mathcal D}p((i,\mathbf{d})\,|\, {\boldsymbol{\omega}}^{0:k-1})\geq p_{\min},$$ for some $p_{\min}>0$;

		\item[(D3)] For a given $\omega=(\omega^k)_k\in\bar{\Omega}$, where $\bar{\Omega}\subseteq\Omega$ such that $\mathbb{P}(\bar{\Omega})=1$, there exists a set $\mathcal{V}(\omega)\subseteq\mathcal{N}\times\mathcal{D}$ such that\\$p((i,\mathbf{d})|\boldsymbol{\omega}^{0:k-1})\geq\Delta>0$ for all $(i,\mathbf{d})\in\mathcal{V}(\omega)$ and $k$, and\\$p((i,\mathbf{d})|\boldsymbol{\omega}^{0:k-1})=0$ otherwise.
		
		\item[(D4)] The block variables are partitioned in $P$ sets $I_1,\ldots,I_P$, where $P$ is the number of available cores, and each core processes a different set $I_j$ of block variables. This implies: $\mathbf{x}_{i^k}^k=\tilde{\mathbf{x}}_i^k$.\medskip
	\end{description}	

  Note that D4 is necessary in order to preserve feasibility of the iterates. This is because the feasible set is nonconvex, and so a convex combination of the two 
feasible vectors $\hat\bx(\tilde \bx^k)$ and $\bx^k$ may not be feasible. However, by construction of the subproblem, 
$\tilde{\bx}^k$ and $\hat\bx(\tilde\bx^k)$ are both feasible and lie in a convex subset of the feasible set, and thus
any convex combination of them is feasible for the original constraints.
	
	We point out that this setting (D4) has proved to be very effective from an experimental point of view even in the absence of nonconvex constraints. 
This has been observed also in \cite{liu2015asyspcd}, and indeed is the setting in which the experiments were performed, despite the convergence proof
drawn up without this arrangement.

	\begin{algorithm}[t]
		\caption{AsyFLEXA: A Global Description}
		\label{alg:global}
		\begin{algorithmic}
			\STATE{\textbf{Initialization:} $k=0$, $\mathbf{x}^0\in\mathcal{X}$,  $\gamma>0$, and $k$.
			}
			\WHILE{a termination criterion is not met}
			\STATE{\texttt{(S.1)}: The random variables $(\underline{i}^k,\underline{\mathbf{d}}^k)$ take realization  $(i,\mathbf{d})$;}%\in \mathcal N\times \mathcal D$;}% with probability $p_{id}^k$; }
			\STATE{\texttt{(S.2)}: Compute $\hat{\mathbf{x}}_{i}(\mathbf{x}^{k-\mathbf{d}})$ [cf.\,(\ref{ncc_2})];} 		
			\STATE{\texttt{(S.3)}: Read   $\bx_{i}^k$;}
			\STATE{\texttt{(S.4)}: Update $\mathbf{x}_i$ by setting  \begin{equation}\label{update_algo_global}\mathbf{x}_{i}^{k+1}=\mathbf{x}^k_{i}+\gamma(\hat{\mathbf{x}}_{i}(\mathbf{x}^{k-\mathbf{d}})-\mathbf{x}_{i}^k);\end{equation}\vspace{-0.4cm}}
			\STATE{\texttt{(S.5)}: Update the iteration counter $k \leftarrow k+1;$} 	\ENDWHILE
			\RETURN $\mathbf{x}^k$
		\end{algorithmic}
	\end{algorithm}

\section{Complexity}	Assumptions A-D in the previous two sections correspond exactly to the assumptions used in the companion paper \cite{companion2} in order to analyze the convergence properties of 
Algorithm 1 therein when applied to the general non convex problem \eqref{ncc_1}; more precisely they correspond to Assumptions  A$^\prime$, B$^\prime$, E, and  C respectively.  Note that Algorithm \ref{alg:global} and Algorithm
AsyFLEXA-NCC in \cite[Section 4]{companion2} are identical with the only exception that {\em in  \cite{companion2} we use a diminishing stepsize while in Algorithm \ref{alg:global} we use a fixed stepsize $\gamma$}. The difference is standard and derives from the necessity to get complexity results; as a part of our complexity result we will also show convergence of this fixed step size variant of  Algorithm 1 in \cite{companion2}.
Besides the rather standard Assumptions A-D, in order to get complexity results we also need that the solution
mapping $\hat{\textbf{x}}_i(\cdot)$ is Lipschitz continuous on $\cal X$.
\smallskip

	\noindent \textbf{Assumption E}
	%\begin{description}[topsep=-2.0pt,itemsep=-2.0pt]
		%\item[  (E1)] 
		$\hat{\textbf{x}}_i(\cdot)$ is Lipschitz continuous with constant 
		$L_{\hat{x}}$ on $\cal X$\medskip
%	\end{description}
\smallskip

\noindent
This assumption is rather mild. While, in order to concentrate on the complexity result,  we postpone its more detailed analysis to  the Appendix, we mention here that, in our setting and supposing $\cal X$ is bounded, it is automatically satisfied if the feasible region of  problem \eqref{ncc_1} is convex or,  in case of non convex constraints, if some  constraint qualifications are satisfied.

We will use the norm of the following quantity as a measure of optimality:
\begin{equation}
M_F(\mathbf{x})=\mathbf{x}-\underset{\mathbf{y}\in\mathcal{K}_1(\mathbf{x}_1)\times\ldots\times\mathcal{K}_N(\mathbf{x}_N)}{\text{arg min}}\{\nabla f(\mathbf{x})^\text{T}(\mathbf{y}-\mathbf{x})+g(\mathbf{y})+\frac{1}{2}\|\mathbf{y}-\mathbf{x}\|^2_2\}
\end{equation}
This is a valid measure of stationarity because $M_F(\mathbf{x})$ is continuous and $M_F(\mathbf{x})=0$ if 
and only if $\mathbf{x}$ is a stationary solution of Problem \eqref{ncc_1}. Note that it is an extension of the optimality measure given in~\cite{RFLEXA}
for Problem~\eqref{ncc_1} but without the nonconvex constraints. We will study the rate of decrease of $\mathbb{E}(\|M_F(\underline{\mathbf{x}}
^k)\|_2^2)$. More precisely, the following theorem gives an upper bound on the number of iterations needed to decrease $
\mathbb{E}(\|M_F(\underline{\mathbf{x}}^k)\|_2^2)$  below a certain chosen value $\epsilon$, provided that the stepsize 
$\gamma$ is smaller than a certain given constant.

\begin{theorem}\label{compl}
	Consider Assumptions A-E and let $\{\mathbf{x}^k\}$ be the sequence generated by Algorithm \ref{alg:global}, with  $\gamma$ such that:
\begin{equation}
0<\gamma\leq\min\left\{\frac{(1-\rho^{-1})}{2(1+L_{\hat{x}}N(3+2\psi))};\frac{c_{\tilde{f}}}{L_f+\frac{\delta\psi'L_f}{2\Delta}}\right\},
\label{eq:compgbound}
\end{equation}
where $\rho>1$ is any given number.
	Define $K_\epsilon$ to be the first iteration such that $\mathbb{E}(\|M_F(\underline{\mathbf{x}}^k)\|_2^2)\leq\epsilon$. Then:
	\begin{equation}\label{eq:compKbound}
	K_\epsilon\leq \frac{1}{\epsilon}\, \frac{4(1+(1+L_B+L_E)(1+L_EL_B\delta\psi'\gamma^2))}{ p_{\text{min}}\gamma(2\Delta(c_{\tilde{f}}-\gamma L_f)-\gamma\delta\psi'L_f)}(F(\mathbf{x}^0)-F^*),
	\end{equation} where $F^*=\underset{\mathbf{x}\in\mathcal{K}}{\min}\,F(\mathbf{x})$, $\psi\triangleq\sum\limits_{t=1}^{\delta}\rho^{\frac{t}{2}}$, and $\psi'\triangleq\sum\limits_{t=1}^{\delta}\rho^t$.\\
\end{theorem}	
\begin{proof}
	See Appendix.
\end{proof}

%\begin{remark} \rm
%Since $\epsilon$ can be arbitrarily small,  Theorem \ref{compl} shows that  Algorithm \ref{alg:global} will produce a %sequence $\{\mathbf{x}^k\}$ such that each of its limit points is stationary for problem \eqref{ncc_1}.
%\end{remark}

\begin{remark} \rm
There is an inverse proportionality  relation between the desired optimality tolerance
satisfaction and the required number of iterations, i.e., $K_\epsilon \propto \frac{1}{\epsilon}$. This is consistent with the 
sequential complexity result for SCA algorithms for nonconvex problems~\cite{RFLEXA} from which it can also be seen that the
constants match when asynchrony is removed. Similarly, it matches   results   on asynchronous methods
for nonconvex problems as appearing in~\cite{lian2015asynchronous,DavisEdmundsUdell}. This suggests that the actual complexity rate
is tight, despite a more complex, and thus difficult to analyze, model.
\end{remark}

\begin{remark} \rm
The bounds  \eqref{eq:compgbound} and \eqref{eq:compKbound} are rather intricate and depend in a complex way on all constants involved. However, it is of particular interest to try to understand how the speedup of the method is influenced by the number of cores used in the computations, all other parameters of the problem being fixed.  
We consider the most common
shared memory architecture and denote by $N_r$ the number of cores.
To perform the analysis we assume that $\gamma$ is ``very small". By this we mean that not only does $\gamma$ satisfy
\eqref{eq:compgbound} but (a) it is such that  variations of $N_r$ will make  \eqref{eq:compgbound} still satisfied by the chosen value of $\gamma$ and (b) the chosen value of $\gamma$  will make all terms involving $\gamma^2$ in 
\eqref{eq:compKbound}  negligible with respect to the other terms. Under these circumstances \eqref{eq:compKbound} reads, for some suitable constant $C_1$,
\[
K_\epsilon \lesssim \frac{1}{\epsilon} \frac{C_1}{\gamma}\] 
%and therefore the speedup when using $N_r$ cores instead of one is
%\[
% \frac{1}{\epsilon} \frac{C_1}{\gamma }
%\frac{1}{N_r}, 
 %\]
 thus implying a linear speedup with the number of cores.  
Note that when $N_r$ increases the value on the right-hand-side of   \eqref{eq:compgbound} will decrease, because 
 we can reasonably assume that  $N_r \approx \delta$, so that if $N_r$ increases $\psi$ and $\psi^\prime$ also increase. Once the right-hand-side of \eqref{eq:compgbound} hits the chosen value of $\gamma$ the analysis above fails and linear speedup is harder to establish.
Therefore we expect that the smaller the chosen $\gamma$  the larger the number of cores for which we can guarantee theoretical linear speedup.
 One should take into account that this type of results is mainly of theoretical interest because
on the one hand  \eqref{eq:compgbound} and \eqref{eq:compKbound} only give  (worst case scenario) upper estimates and, on the other hand, in practice one would prefer to use larger values of $\gamma$. 
\end{remark}
	
An interesting result can be easily derived from Theorem \ref{compl} if we consider a synchronous version of Algorithm 1.
\begin{corollary}
	Consider a synchronous version of Algorithm 1, i.e. $\mathbf{x}^{k-\mathbf{d}}=\mathbf{x}^k\,\forall\,k$, under assumptions A-D, where at each iteration a block is randomly picked to be updated according to a uniform distribution. Choose $\gamma$ such that:
	\begin{equation}
	0<\gamma\leq\min\left\{\frac{(1-\rho^{-1})}{2(1+3L_{\hat{x}}N)};\frac{c_{\tilde{f}}}{L_f}\right\}.
	\label{eq:compgbound2}
	\end{equation}
	Then:
	\begin{equation}\label{eq:compKbound2}
	K_\epsilon\sim\mathcal{O}\left(\frac{N^3}{\epsilon}\right).
	\end{equation}\\	
\end{corollary}	
\begin{proof}
	Since we consider a synchronous scheme, we have $\mathcal{D}=\{\mathbf{0}\}$, which implies $\delta=0$ and $\psi=\psi'=0$. The uniform random block selection implies that $p_\text{min}=\Delta=\frac{1}{N}$. Substituting these values in \eqref{eq:compgbound} and \eqref{eq:compKbound} we get respectively \eqref{eq:compgbound2} and \eqref{eq:compKbound2}. 
\end{proof}
The Corollary states, as expected, that in a synchronous implementation of our algorithm, the iteration complexity does not depend (ideally) on the number of cores running. This of course comes from the fact that in this setting the convergence speed of the algorithm is no longer affected by the use of
 old information $\tilde{\mathbf{x}}^k$.

	\section{Numerical Results}\label{numerical results}
	In this section we test our algorithm on  LASSO problems and  on a nonconvex sparse learning problem and compare it to 
	state-of-art methods.  In particular, we test   the diminishing-stepsize version of the method proposed and studied in 
	the companion paper \cite{companion2}. Results for the fixed-stepsize version of the method whose complexity has been 
	studied in the previous section are not reported. Indeed, as usual for these methods, if the theoretical stepsize 
	\eqref{eq:compgbound} is used, the algorithms simply do not make any practical progress towards optimality  in a 
	reasonable number of iterations. If, on the other hand, we disregard the bound \eqref{eq:compgbound} that, we recall, is 
	an upper bound, and use a ``large'' stepsize (in particular we found the value of  0.95 practically effective), the numerical 
	results are essentially the same as those obtained with the diminishing stepsize version for which we report the results. 
	Since this latter version, as shown in \cite{companion2}, has theoretical convergence guarantee for the chosen stepsize 
	rule, we present the results only for this version of the method.

	\subsection{Implementation } \label{CODES}	
	All tested codes have been written in C++  using  Open- MP. The algorithms were tested on the Purdue's Community 
	Cluster  Snyder on a machine with two 10-Core Intel Xeon-E5 processors (20 cores in total) and 256 GB of RAM.
	
	\noindent$\bullet$\textbf{AsyFLEXA}: We tested Algorithm 1 from \cite{companion2}, always setting $N=n$, i.e. we considered 
	scalar subproblems. These subproblems,  for the choice of $\tilde{f}_i$ that we used and that we specify below for each 
	class of problems, can be solved in closed form using the soft-thresholding operator \cite{beck_teboulle_jis2009}. For the 
	stepsize sequence $\{\gamma^k\}_{k\in\mathbb{N}_+}$  we used the rule 
	$\gamma^{k+1}=\gamma^k(1-\mu\gamma^k)$ with $\gamma^0=1$ and $\mu=10^{-6}$.

		We compared  AsyFLEXA with the following state-of-the-art asynchronous schemes: {A{\scriptsize{SY}}SPCD} and {ARock}. We underline that the theoretical stepsize rules required for the convergence  of {A{\scriptsize{SY}}SPCD} and {ARock} lead to practical non-convergence, since they prescribe extremely small stepsizes. For both algorithms we made several attempts to identify  practical rules that bring the best numerical behaviour, as detailed below. 
		
		\noindent$\bullet$ \textbf{A{\scriptsize{SY}}SPCD}: This is the asynchronous parallel stochastic proximal gradient 
		coordinate descent algorithm for the minimization of convex and nonsmooth objective functions presented in 
		\cite{liu2015asyspcd}. Since the algorithm was observed to perform poorly if the stepsize $\gamma$ is chosen 
		according to \cite[Th. 4.1]{liu2015asyspcd}$-$the  value guaranteeing theoretical  convergence$-$in our experiments 
		we used $\gamma=1$, which violates \cite[Th. 4.1]{liu2015asyspcd} but was effective on our test problems. Note also that this is the value actually used in the numerical tests reported in  \cite{liu2015asyspcd}. We underline that in order to implement the algorithm it is also required to estimate some Lipschitz-like constants.
		\\
		$\bullet$ \textbf{ARock}: ARock \cite{peng2015arock} is an asynchronous parallel algorithm proposed to compute 
		fixed-points of a nonexpansive operator.  Thus it can be  used to solve convex optimization problems.  The algorithm 
		requires a stepsize and 
		the knowledge of the Lipschitz constant of $\nabla f$. Here, again, the use of stepsizes that guarantee theoretical convergence leads to very poor practical performance. In this case we found that the most practical, effective version 
		of the method could be obtained by using for the stepsize  the same rule employed in our AsyFLEXA, with a safeguard that guarantees that the stepsize never becomes smaller than 0.1. 
%		
%		for the former we used the same rule as  for AsyFLEXA while we 
%		estimated the latter using  a backtracking procedure (Note that we did not   include in the CPU time the time spent to 
%		compute the aforementioned constants).
\\

We also tested  the stochastic version of Asynchronous PALM \cite{davis2016asynchronous}, an asynchronous version of the Proximal Alternating Linearized Minimization (PALM) algorithm for the minimization of nonsmooth and nonconvex objective functions. We did not report the results, because  its   practical implementation    (using  unitary stepsize rather than the one guaranteeing convergence, for the reasons explained above) basically coincides with the one of A{\scriptsize{SY}}SPCD.
		
Before reporting the numerical results it is of interest to briefly contrast these algorithms in order to better highlight some interesting properties of AsyFLEXA.

\begin{enumerate}
\item In all tests we  partitioned the variables  among the cores, with only one core per partition.   Therefore, each core is in charge of 
updating its assigned variables and no other core will update those variables. There is consensus in the literature that this configuration 
brings better results experimentally in shared memory architectures, see e.g. \cite{Hogwild!,liu2015asyspcd}. Furthermore, if one is 
considering a message passing architecture, this is  actually the only possible choice.  It is then interesting to note that this choice is fully 
covered by our theory, while the theory of ARock \cite{peng2015arock} requires that each core has access to all variables and therefore can 
not handle it. The theory supporting A{\scriptsize{SY}}SPCD \cite{liu2015asyspcd} requires that at each iteration a variable is updated, with 
all variables selected with {\em equal} probability. The equal probability requirement in the partitioned configuration can not be  
theoretically excluded, but is totally unlikely: it would require that all cpu are identical and that all subproblems require exactly the same 
amount of time  to be processed.  In any case, even not considering the partitioned variables issue, the practical choice for the stepsizes 
adopted are not covered by the theory in \cite{liu2015asyspcd} and \cite{peng2015arock}. We believe this clearly shows our analysis to be 
robust and well matched to  practical computational architectures.
\item It is  of interest to note that we run experiments also on {\em nonconvex} problems. Our algorithm has theoretical convergence
guarantees for nonconvex problems while neither ARock nor A{\scriptsize{SY}}SPCD has any such guarantees.
\item Both A{\scriptsize{SY}}SPCD and ARock (and also PALM) theoretically require, for their  implementation, the estimation of some Lipschitz constants of the problem functions. Although in our test problems the estimations could be reasonably done, this requirement in general can be extremely time consuming if at all possible on other problems. 
\item One last difference worth mentioning is that contrary to AsyFLEXA and ARock, A{\scriptsize{SY}}SPCD and PALM require a memory lock of the variable(s) being updated while the solution of the corresponding  subproblems are computed. In the 
partitioned configuration we adopted this is not an issue, because nobody else can update the variables being updated by a given core. However, when other configurations are chosen this can become a source of delays, especially if  the subproblems are complex and take some time to be solved. This could easily happen, for example, if the subproblems involve more than one variable  or if they are nonconvex, as those that could appear in PALM
\end{enumerate}

We are now ready to illustrate the numerical results.

\subsection{LASSO} \label{LASSO}
	Consider the LASSO problem, i.e. Problem \eqref{ncc_1}, with\\$f(\mathbf{x}) =  \frac{1}{2}\|\mathbf{Ax} - \mathbf{b}\|^2_2$,  $G(\mathbf{x}) = \lambda \cdot\|\mathbf{x}\|_1$, and $\mathcal{X}=\mathbb{R}^n$, with  $\mathbf{A}\in\mathbb{R}^{m\times n}$ and $\mathbf{b}\in\mathbb{R}^{m}$.
	In all implementations of the algorithms,  we precomputed the matrix $\mathbf{A}^\text{T}\mathbf{A}$ and the vector $\mathbf{A}^\text{T}\mathbf{b}$ offline.  In all the experiments,  the starting point of all the algorithms  was set to the zero vector. For  problems in which the optimal value $F^*$ is not known, in order to compute the relative error we estimated $F^*$ by running a synchronous version of AsyFLEXA until variations in the objective value were undetectable.
	
	For AsyFLEXA we set
	$\tilde{f}_i(x_i;\mathbf{x}^{k-\mathbf{d}})=f(x_i;\mathbf{x}^{k-\mathbf{d}}_{-i})+\frac{\tau^k}{2}(x_i-x_i^{k-d_i})^2$, where $\mathbf{x}_{-i}$ denotes the vector obtained from $\mathbf{x}$ by deleting the block $\mathbf{x}_i$. The sequence $\{\tau^k\}_{k\in\mathbb{N}_+}$, shared among all the cores,  was updated   every $n$ iterations, according to the  heuristic proposed for  FLEXA \cite{FLEXA}.
		
		We run tests on LASSO problems generated according to three different methods;  all curves reported  are averaged over five independent random  problem realizations.	
		\\
		
			\noindent \emph{Gaussian problems \cite{liu2015asyspcd}:}   In this setting we generated the LASSO problem according to the procedure used in \cite{liu2015asyspcd}; where the matrix $\mathbf{A}$ has samples taken from a Gaussian $\mathcal{N}(0,1)$ distribution, $\bar{\mathbf{x}}\in\mathbb{R}^n$ is a sparse vector with $s$ nonzeros and $\mathbf{b}=\mathbf{A}\bar{\mathbf{x}}+\mathbf{e}$, with $\mathbf{e}\in\mathbb{R}^m\sim\mathcal{N}(0,\sigma^2)$. In particular we chose $m=20000$, $n=40000$, $s=40$ and $\sigma=0.01$. For the regularization parameter we used the value suggested in \cite{liu2015asyspcd}: $\lambda=20\sqrt{m\log(n)}\sigma$.
			In \figurename~\ref{fig:wrig20k40k} we plot the relative error  on the objective function versus the CPU time, using  2 and 20 cores ($(c)=$ cores). The figure shows that when increasing the number of cores, all the algorithms converge quickly to the solution with comparable performances. In this particular setting, and contrary to what happens in all other problems,   A{\scriptsize{SY}}SPCD has a slightly better behavior than AsyFLEXA but it does not have convergence guarantees. 
			
%		\textcolor{red}{	LORIS, PLEASE EXPLAIN BETTER. I DO NOT KNWO WHAT YOU MEAN HERE.}
%			This is  is the fact that our algorithm requires the solution of subproblems in which the old information $\mathbf{x}
%^{k-\mathbf{d}}$ is involved. On the other hand, A{\scriptsize{SY}}SPCD uses the old information $\mathbf{x}^{k-
%\mathbf{d}}$ only in the evaluation of the gradient of $f$ and not in the rest of the computations needed for the updating 
%of $\mathbf{x}^k$. This fact appears to speed up the convergence of A{\scriptsize{SY}}SPCD in this particular setting, but 
%in more complex problems it is not clear how to update $\mathbf{x}^k$ only considering that the old information will 
%affect only the computation of $\nabla f$.    

			In order to quantify the scalability of the algorithms, in \figurename~\ref{fig:sp_wrig20k40k} we plot the  speedup 
			achieved by each of the algorithms versus the number of cores  (of course we run all 
			algorithms also for values between $c=1$ and $c=20$, more precisely for $c= 1, 2, 4, 8, 10, 20$).
		We defined the speedup as the ratio between the runtime on a single core and the runtime on multiple cores.
			The runtimes we used are the CPU times needed to reach a relative error strictly less than $10^{-4}$. The figure 
			shows that all the algorithms obtain a good gain in the performances by increasing the number of cores which is not 
			far from the ideal speedup.
			\begin{figure}[ht]
				\begin{minipage}{0.48\textwidth}
				\centering
				\includegraphics[width = 6cm]{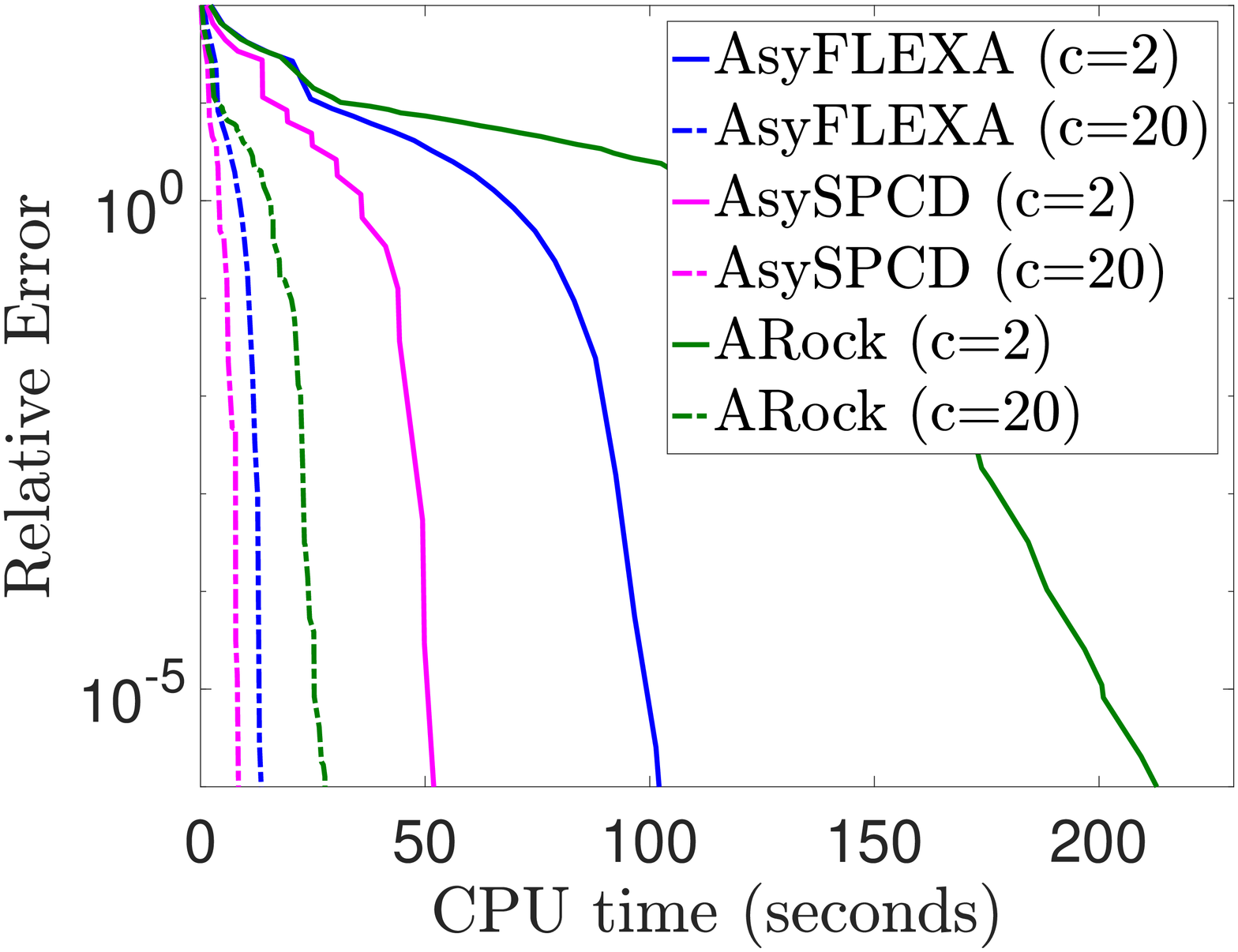}
				\caption{LASSO: comparison in terms of relative error versus CPU time (in seconds) for Liu and  Wright's problems 
				\cite{liu2015asyspcd}}
				\label{fig:wrig20k40k}
\end{minipage}\hfill
\begin {minipage}{0.48\textwidth}
				\centering
				\vspace{-0.4cm}
				\includegraphics[width = 6cm]{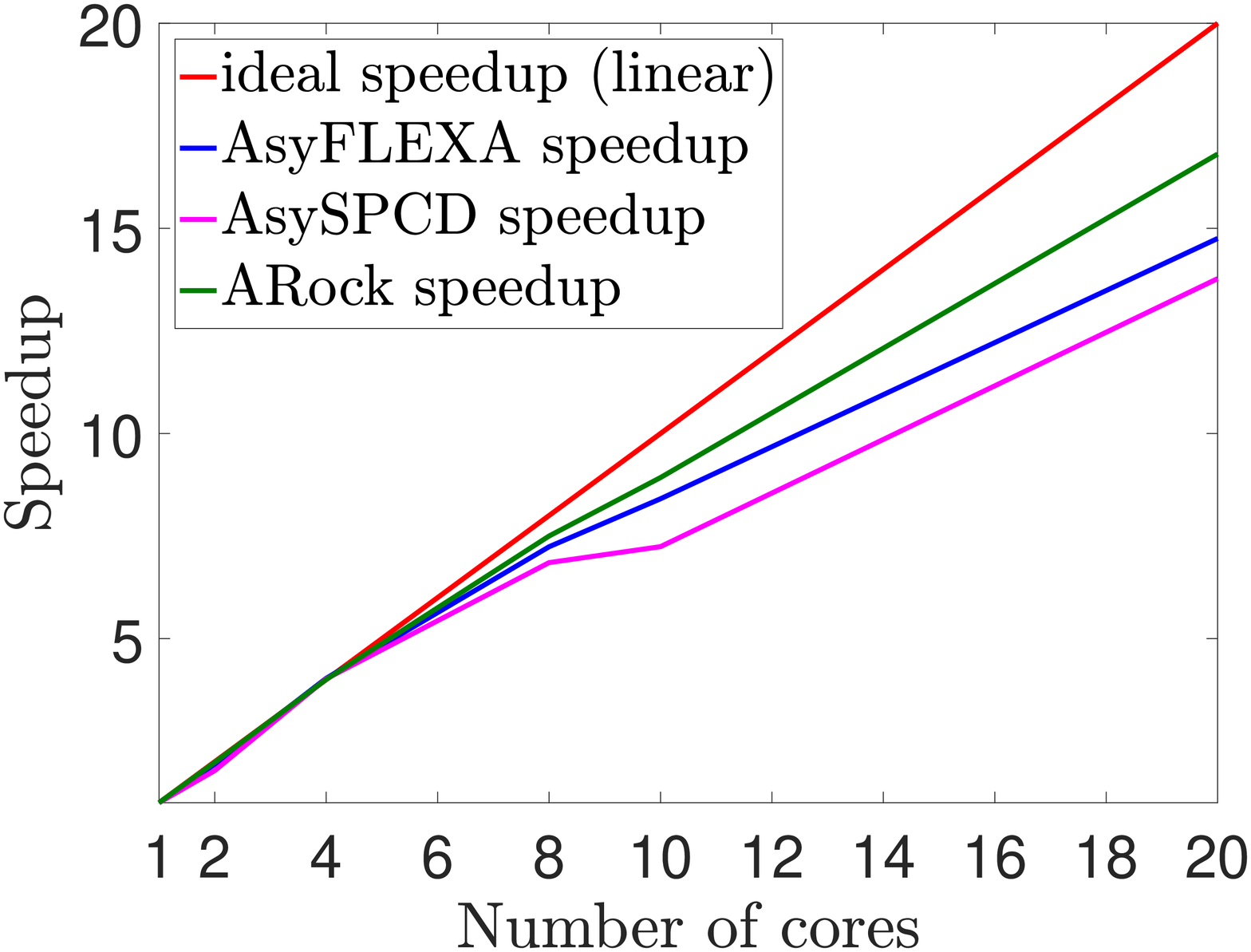}
				\caption{LASSO: speedup of the tested algorithms for Liu and  Wright's problems \cite{liu2015asyspcd}}
				\label{fig:sp_wrig20k40k}
			\end{minipage}
			\end{figure}
		\\
		
		\noindent \emph{Nesterov's problems \cite{nesterov2012gradient}:}   Here we generated a LASSO problem using the 
		random generator proposed by Nesterov in \cite{nesterov2012gradient}, which permits us to control the sparsity of the 
		solution. We considered a problem with $40000$ variables and matrix  $\mathbf{A}$ having $20000$ rows, and set $
		\lambda=1$; the percentage of nonzero in the solution is $1\%$.
		In \figurename~\ref{fig:sp_nest20k40k} we plot the relative error  on the objective function (note that for  Nesterov's 
		model the optimal solution is known) versus the CPU time, using  2 and 20 cores. The figure clearly shows that 
		AsyFLEXA significantly outperforms all the other algorithms on these problems. Moreover, the empirical convergence 
		speed  significantly  
		increases with the number of cores, which  instead  is not observed for the other algorithms, see 
		\figurename~\ref{fig:sp_nest20k40k}, where we only report data for AsyFLEXA, given that the other algorithms do not reach the prefixed   threshold error value of $10^{-4}$ in one hour of computation time.
					\begin{figure}[ht]
						\begin{minipage}{0.48\textwidth}
							\centering
							\includegraphics[width = 6cm]{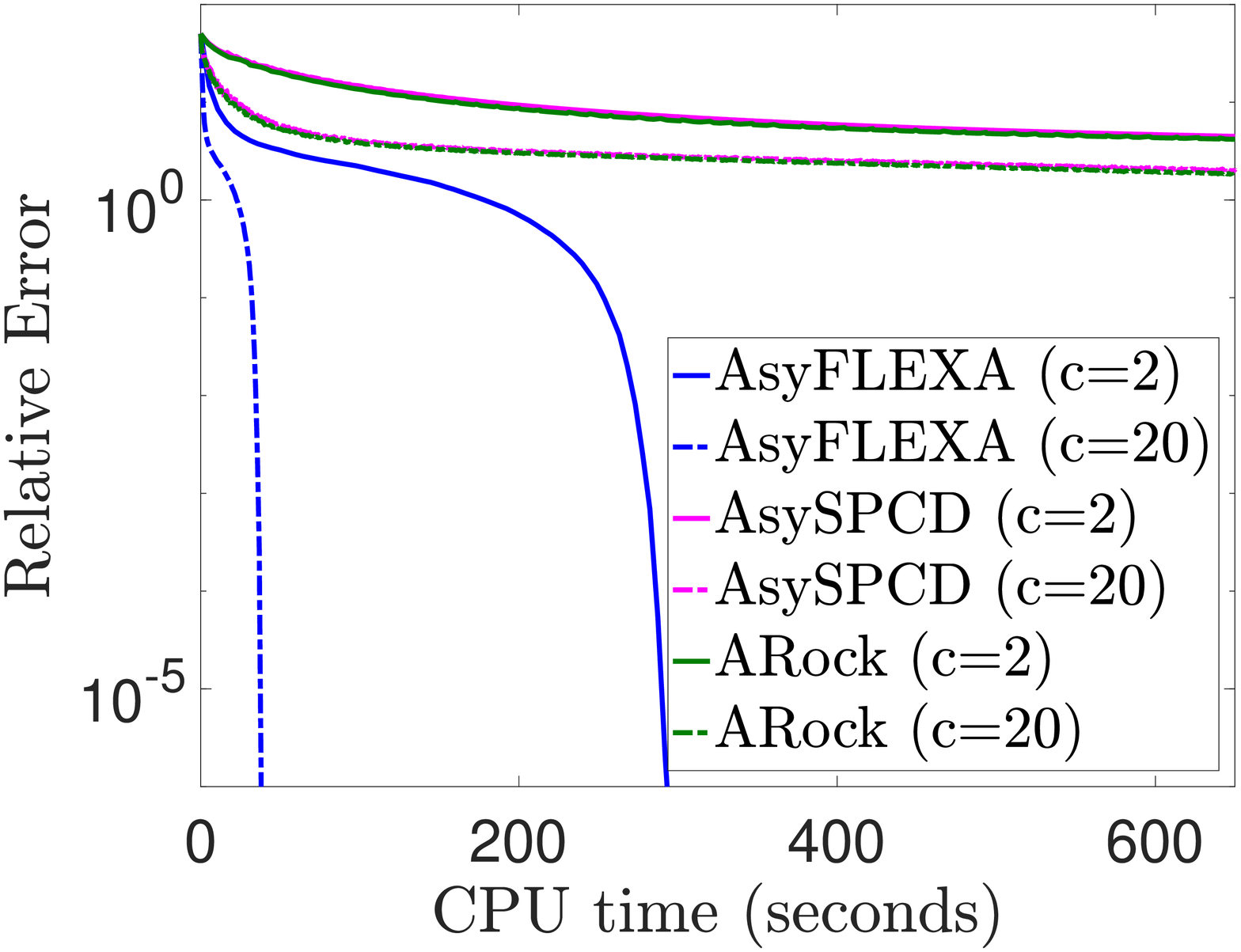}
							\caption{LASSO: comparison in terms of relative error versus CPU time (in seconds) for Nesterov's problems \cite{nesterov2012gradient}}
							\label{fig:nest20k40k}
						\end{minipage}\hfill
						\begin {minipage}{0.48\textwidth}
						\centering
						\vspace{-0.4cm}
						\includegraphics[width = 6cm]{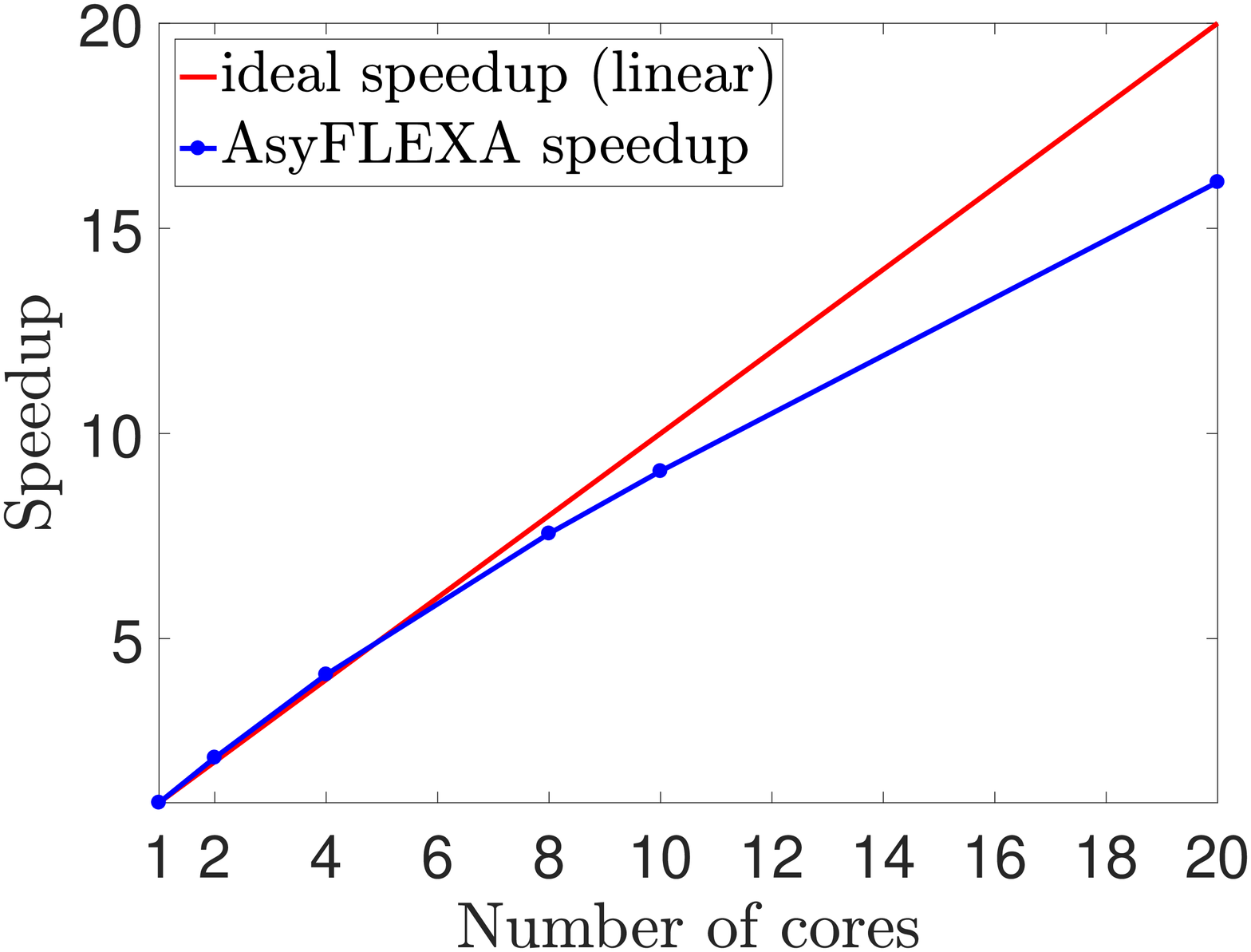}
						\caption{LASSO: speedup of AsyFLEXA for  Nesterov's problems \cite{nesterov2012gradient}}
						\label{fig:sp_nest20k40k}
					\end{minipage}
				\end{figure}
		
		\noindent \emph{Gondzio's problem:} We generated these LASSO problems using the generator proposed by Gondzio in \cite{gondzio2013second}. The key feature of this generator is the possibility of choosing the condition number of $\mathbf{A}^\text{T}\mathbf{A}$, and we set it to $10^4$.  We generated a matrix $\mathbf{A}$ with $2^{14}$ rows and $\lceil 1.01n\rceil$ columns, where the ceiling function $\lceil \cdot\rceil$ returns the smallest integer greater than or equal to its argument. The sparsity in the solution is 0.1\% and we set $\lambda=1$. \figurename~\ref{fig:gondzio16384k2} shows the relative error with respect to the CPU time for the different algorithms we tested, when using 2 and 16 cores.   
		\figurename~\ref{fig:sp_gondzio16384k2} that shows the speedup achieved by our algorithm on these problems
		(in this case we run the algorithm for $c= 1,2,4,8,16$).  
We see that the behavior of the algorithm is qualitatively very similar to the one obtained for the previous set of problems.	
		
\begin{figure}[ht]
	\begin{minipage}{0.48\textwidth}
		\centering
		\includegraphics[width = 6cm]{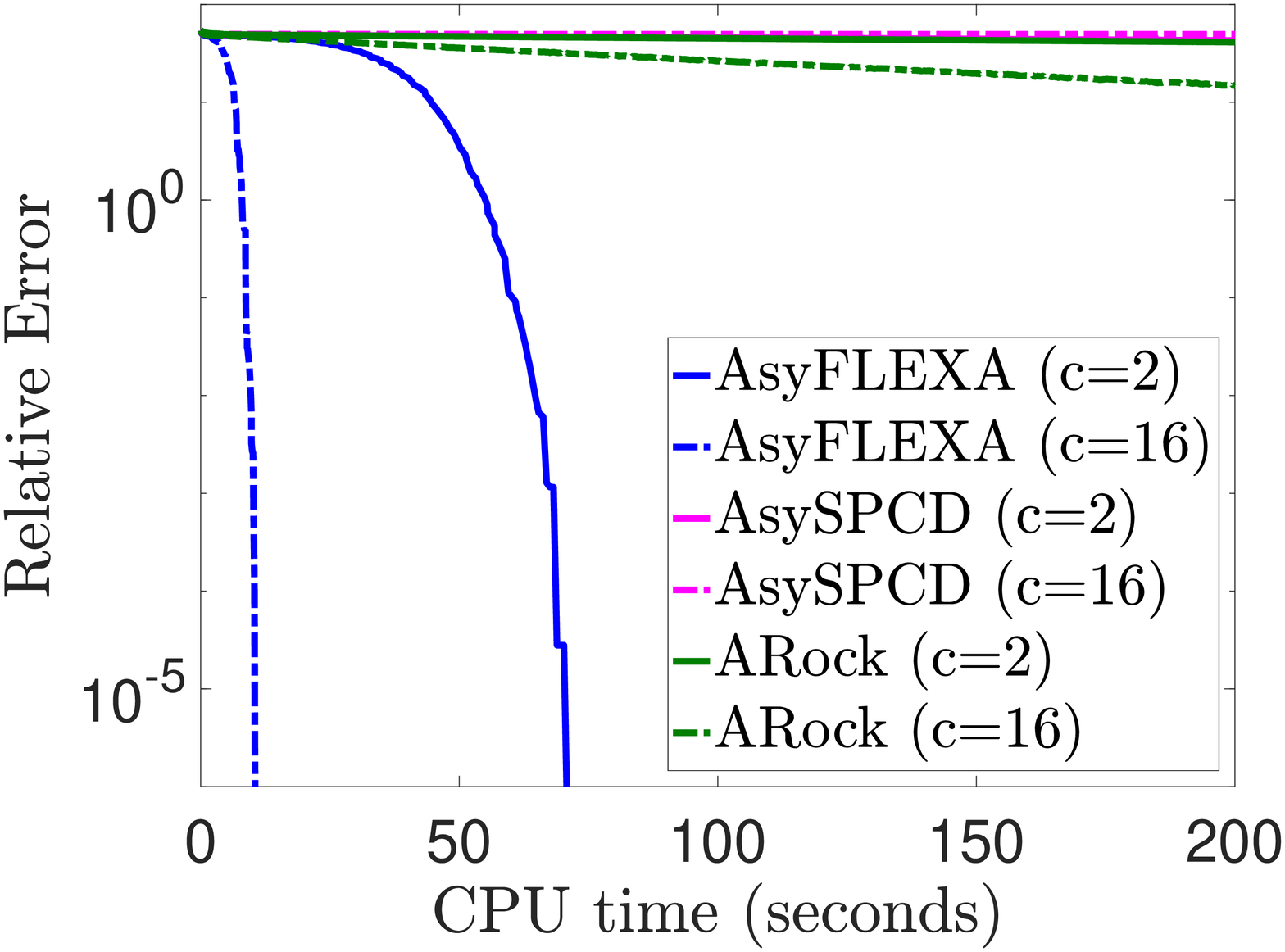}
		\caption{LASSO: comparison in terms of relative error versus CPU time (in seconds) for  Gondzio's problems \cite{gondzio2013second}}
		\label{fig:gondzio16384k2}
	\end{minipage}\hfill
	\begin {minipage}{0.48\textwidth}
	\centering
	\vspace{-0.4cm}
	\includegraphics[width = 6cm]{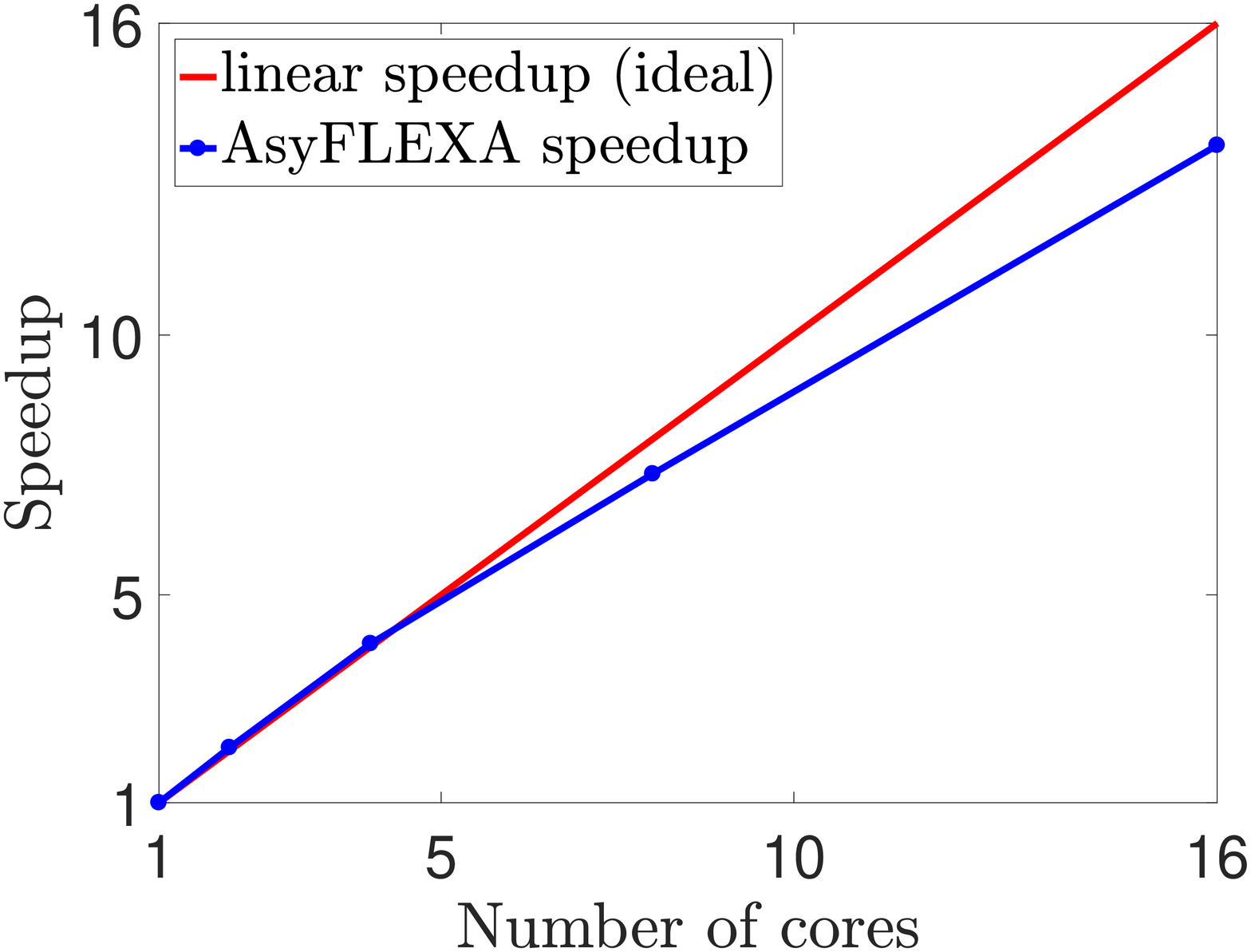}
	\caption{LASSO: speedup of AsyFLEXA for  Gondzio's problems \cite{gondzio2013second}}
	\label{fig:sp_gondzio16384k2}
\end{minipage}
\end{figure}
		
		\subsection{Nonconvex Sparse Learning}\label{subsec:nonconvex_quadratic}\vspace{-0.1cm}
		We consider now the following nonconvex instance of problem  \eqref{ncc_1}:
		\begin{equation}\label{ncc}
		\begin{array}{cl}
		\underset{\mathbf{x}\in\mathbb{R}^n}{\text{minimize}} & F(\mathbf{x})\triangleq \underset{=m(\mathbf{x})}{\underbrace{\|\mathbf{A}\mathbf{x}-\mathbf{b}\|^2_2}}+\lambda H(\mathbf{x}),
		\end{array}
		\end{equation}
		where $\mathbf{A}\in\mathbb{R}^{m\times n}$, $\mathbf{b}\in\mathbb{R}^m$,  $\lambda>0$ and $H(\mathbf{x})$ is a regularizer used to balance the amount of sparsity in the solution. \eqref{ncc} is a standard formulation used for doing regression in order to recover a sparse signal from a noisy observation vector $\mathbf{b}$ \cite{zhang2014minimization}. Common choices for the regularizer $H$ are surrogates of the $l_0$ norm, as the frequently used $l_1$ norm (in this case we recover the LASSO problem of the previous section). However, it is known that nonconvex surrogates of the $l_0$ norm can provide better solutions than the $l_1$ norm, in terms of balance between compression and reconstruction error, see e.g. \cite{le2015dc}. For this reason we tested here two nonconvex regularizer functions, that we already used in \cite{scutari2016parallel}: the exponential one $H(\mathbf{x})=H_\text{exp}(\mathbf{x})=\sum\limits_{i=1}^nh_\text{exp}(x_i)$, with $h_\text{exp}(x_i)=1-e^{-\theta_\text{exp}|x|}$ and $\theta_\text{exp}>0$, and the logarithmic one $H(\mathbf{x})=H_\text{log}(\mathbf{x})=\sum\limits_{i=1}^nh_\text{log}(x_i)$, with $h_\text{log}(x_i)=\frac{\log(1+\theta_\text{log}|x|)}{\log(1+\theta_\text{log})}$ and $\theta_\text{log}>0$.
		Note that Problem \eqref{ncc} does not immediately appear to be in the format needed by our algorithm, i.e. the sum of a smooth term and of a possibly nondifferentiable convex one. But we can put the problem in this form, as explained next.
		In both cases the function $h$ possesses a DC structure that allows us to rewrite it as 
		\begin{equation}
		h(x) = \underset{\triangleq h^+(x)}{\underbrace{\eta(\theta)|x|}}-\underset{\triangleq h^-(x)}{\underbrace{\eta(\theta)|x|-h(x)}},
		\end{equation}   
		with $\eta_\text{exp}(\theta_\text{exp})=\theta_\text{exp}$ and $\eta_\text{log}(\theta_\text{log})=\frac{\theta_\text{log}}{\log(1+\theta_\text{log})}$. 
		It is easily verified that $h^+$ is convex  and that, slightly more surprisingly, $h^-$ is continuously differentiable with a Lipschitz gradient \cite{le2015dc}. Given these facts, it is now easily seen that problem \eqref{ncc} can be put in the required format by setting $f \triangleq h^-$ and $G\triangleq  m+ h^+$. With this in mind,
we can implement  AsyFLEXA to solve problem \eqref{ncc}  by considering scalar blocks ($N=n$) and the following formulation for the best-response $\hat{x}_i$, for $i=1,\ldots,n$:
		\begin{equation}
		\begin{array}{l}
			\hat{x}_i(\mathbf{x}^{k-\mathbf{d}})=\\
			\underset{x_i\in\mathbb{R}}{\text{arg min}}\left\{f(x_i;\mathbf{x}^{k-\mathbf{d}}_{-i})-\lambda\nabla h^-(x_i^{k-d_i})^\text{T}(x-x_i^{k-d_i})+\lambda h^+(x_i)+\frac{\tau^k}{2}(x_i-x_i^{k-d_i})^2\right\},
		\end{array}
		\label{ncc_subpr}
		\end{equation}
		 Note that, once again, the solution of  \eqref{ncc_subpr} can be computed in closed-form through the soft-thresholding operator \cite{beck_teboulle_jis2009}. For the positive constant $\tau^k$ we use again the same heuristic rule
		 used for the LASSO problems. 
		 
In order to generate the 5 random instances of the problem we must specify how we generate the quadratic function $m(\cdot)$. 
In order to favour  	A{\scriptsize{SY}}SPCD and ARock we use the Liu and Wright's generator used  for the first set of LASSO problems discussed above. 	
		In particular, we generated the underlying sparse linear model according to: $\mathbf{b}=\mathbf{A}\bar{\mathbf{x}}+\mathbf{e}$ where $\mathbf{A}$ has 20000 rows (with values normalized to one) and 40000 columns. $\mathbf{A}$, $\bar{\mathbf{x}}$ and $\mathbf{e}$ have i.i.d. elements coming from a Gaussian $\mathcal{N}(0,\sigma^2)$ distribution, with $\sigma=1$ for $\mathbf{A}$ and $\bar{\mathbf{x}}$, and $\sigma=0.1$ for the noise vector $\mathbf{e}$. To impose sparsity on $\bar{\mathbf{x}}$, we randomly set to zero 95\% of its component.
		
		Since \eqref{ncc}  is nonconvex, we compared the performance of the algorithms using the following   distance to  stationarity:  $\|\hat{\mathbf{x}}(\mathbf{x}^k)-\mathbf{x}^k\|_{\infty}$. 	Note that this is a valid stationarity measure: it is continuous and $\|\hat{\mathbf{x}}(\mathbf{x}^\ast)-\mathbf{x}^\ast\|_{\infty}=0$ if and only if  $\mathbf{x}^\ast$ is a stationary solution of \eqref{ncc}.
		In Figure \ref{fig:merit_ncc} we plot the  stationarity measure versus the CPU time, for all the algorithms using 2 and 20 cores, for problem \eqref{ncc} where $H(\mathbf{x})=H_\text{log}(\mathbf{x})$ with $\theta_\text{log}=20$; the curves are averaged over $5$ independent realizations.  All the algorithms were observed to converge to the same stationary solution of (\ref{ncc}) even if, we recall,    A{\scriptsize{SY}}SPCD and ARock have no formal proof of convergence in this nonconvex setting. The figure shows that, even in the nonconvex case,  AsyFLEXA has good performances and actually behaves better that all the other algorithms while being also guaranteed to converge.
				\begin{figure}[ht]
					\centering
					\includegraphics[width = 6cm]{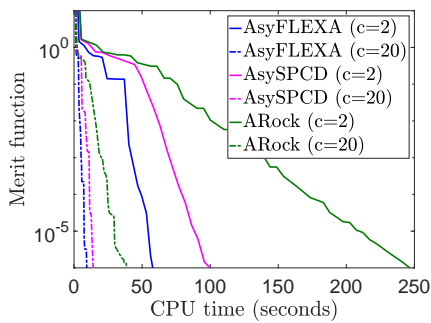}
					\caption{Nonconvex problem: comparison in terms of relative error versus CPU time (in seconds)}
					\label{fig:merit_ncc}
				\end{figure}
As a side issue it is interesting to illustrate the utility of our nonconvex regularizers with respect to the more standard $l_1$ norm regularizer.				
		In \figurename~\ref{fig:err_ncc} and \figurename~\ref{fig:nnz_ncc} we plot respectively the Normalized Mean Square Error (NMSE) (defined as: NMSE$(\mathbf{x})=\|\mathbf{x}-\bar{\mathbf{x}}\|^2_2/\|\bar{\mathbf{x}}\|^2_2$) and the percentage of nonzeros obtained by solving the aforementioned problem with AsyFLEXA for different values of the regularization parameter $\lambda$ and for different surrogates of the $l_0$ norm: the $l_1$ norm, the exponential function ($\theta_\text{exp}=20$) and the logarithmic function ($\theta_\text{log}=20$). AsyFLEXA is terminated when the merit function goes below $10^{-4}$ or after $100n$ iterations. These two figures interestingly show that the two nonconvex regularizers obtain their lowest NMSE for a value of $\lambda$ which corresponds to a good compression result, close to the number of nonzeros of the original signal. On the other side, the $l_1$ norm attains its lowest NMSE by reconstructing the signal with more of 15\% of nonzeros, much more than the original 5\%. In order to get close to the desired number of nonzeros in the reconstruction with the $l_1$ norm, it is necessary to increase the value of $\lambda$, but this leads to a worsening in terms of NMSE of at least two times with respect to the minimum NMSE attainable.
		
		\begin{figure}[ht]
			\begin{minipage}{0.48\textwidth}
				\centering
				\includegraphics[width = 6cm]{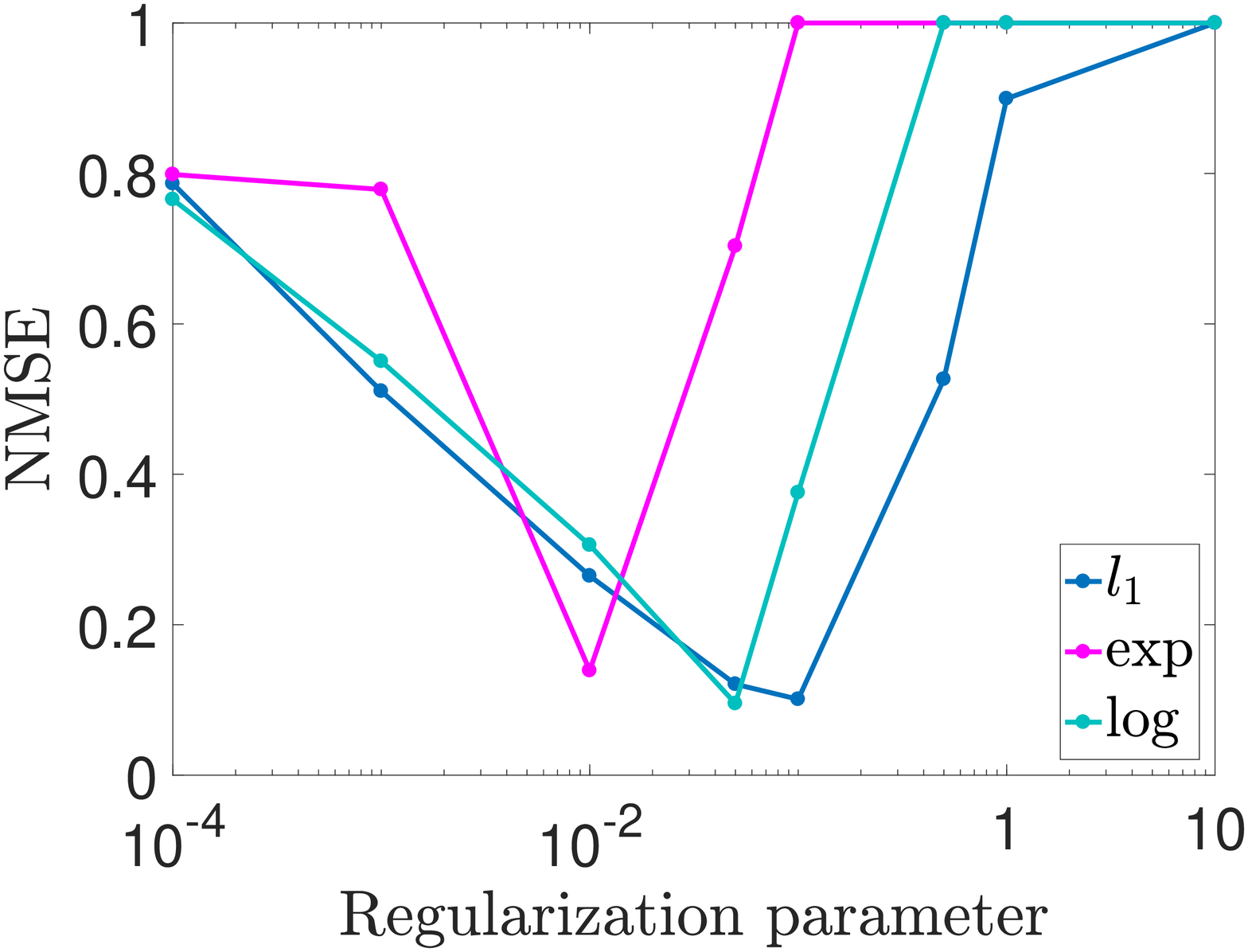}
				\caption{Nonconvex problem: NMSE for different values of the regularization parameter $\lambda$ and for different surrogates of the $l_0$ norm}
				\label{fig:err_ncc}
			\end{minipage}\hfill
			\begin {minipage}{0.48\textwidth}
			\centering
			\vspace{0.2cm}
			\includegraphics[width = 6cm]{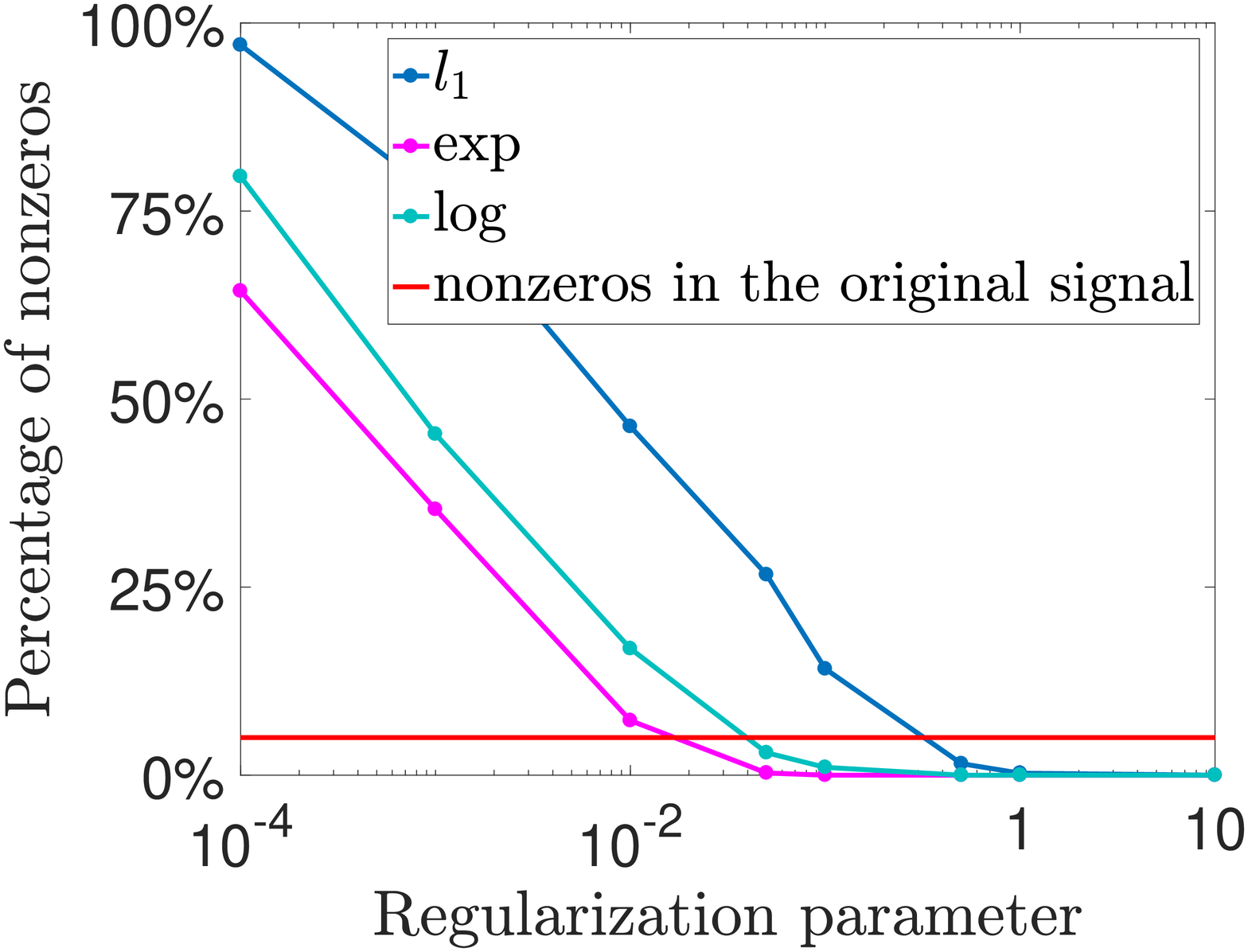}
			\caption{Nonconvex problem: percentage of nonzeros in the solution for different values of the regularization parameter $\lambda$ and for different surrogates of the $l_0$ norm}
			\label{fig:nnz_ncc}
		\end{minipage}
	\end{figure}
	   
\medskip

A full understanding of the numerical behavior of our algorithm  and its comparison with existing alternatives  certainly needs further numerical tests, on both larger and more complex 	 problems. But the substantial results already reported here seem to clearly indicate that our method is reliable and robust and capable to efficiently solve problems that other methods cannot tackle.   The improved performance of AsyFLEXA with respect to A{\scriptsize{SY}}SPCD and ARock on some classes of problems seems due, in our experience, to the much wider flexibility we have in the choice of the model $\tilde f$ which is not linked in any way to Lipschitz constants of the problem functions.  
	   
\section{Conclusions}
Leveraging the new theoretical framework we introduced in the companion paper \cite{companion2}, in this work, we studied   the convergence rate of a novel parallel asynchronous algorithm for the minimization of the sum of a nonconvex smooth function and a convex nonsmooth one subject to nonconvex constraints. Furthermore, we presented numerical results on convex and nonconvex problems showing that our method is widely applicable and   outperforms asynchronous state-of-the-art-schemes in at least 
some classes of important problems.

\section*{Acknowledgments} We very gratefully acknowledge useful discussions with Daniel Raph on stability issues related to Proposition 2.
		
\bibliographystyle{plain}
\addcontentsline{toc}{chapter}{Bibliography}
%\bibliography{NewRefs,scutari_facchinei_refs,Surbib,Bibliography}
\bibliography{biblio}

\begin{thebibliography}{10}

\bibitem{beck_teboulle_jis2009}
A.~Beck and M.~Teboulle.
\newblock {A fast iterative shrinkage-thresholding algorithm for linear inverse
  problems}.
\newblock {\em SIAM Journal on Imaging Sciences}, 2(1):183--202, Jan. 2009.

\bibitem{Bertsekas_Book-Parallel-Comp}
D.~P. Bertsekas and J.~N. Tsitsiklis.
\newblock {\em Parallel and distributed computation: numerical methods},
  volume~23.
\newblock Prentice hall Englewood Cliffs, NJ, 1989.

\bibitem{companion2}
Loris Cannelli, F.~Facchinei, V.~Kungurtsev, and G.~Scutari.
\newblock Asynchronous parallel algorithms for nonconvex big-data optimization
  - part i: Model and convergence.
\newblock {\em Submitted to Mathematical Programming}, 2016.

\bibitem{davis2016asynchronous}
D.~Davis.
\newblock The asynchronous palm algorithm for nonsmooth nonconvex problems.
\newblock {\em arXiv preprint arXiv:1604.00526}, 2016.

\bibitem{DavisEdmundsUdell}
D.~Davis, B.~Edmunds, and M.~Udell.
\newblock The sound of apalm clapping: Faster nonsmooth nonconvex optimization
  with stochastic asynchronous palm.
\newblock {\em arXiv preprint arXiv:1606.02338}, 2016.

\bibitem{FLEXA}
F.~Facchinei, G.~Scutari, and S.~Sagratella.
\newblock Parallel selective algorithms for nonconvex big data optimization.
\newblock {\em IEEE Transactions on Signal Processing}, 63(7):1874--1889, 2015.

\bibitem{gondzio2013second}
K.~Fountoulakis and J.~Gondzio.
\newblock A second-order method for strongly convex$\backslash$ ell
  \_1-regularization problems.
\newblock {\em Mathematical Programming}, 156(1-2):189--219, 2016.

\bibitem{hong2014distributed}
M.~Hong.
\newblock A distributed, asynchronous and incremental algorithm for nonconvex
  optimization: An admm based approach.
\newblock {\em arXiv preprint arXiv:1412.6058}, 2014.

\bibitem{le2015dc}
H.~A. Le~Thi, T~P. Dinh, H.~M. Le, and X.~T. Vo.
\newblock Dc approximation approaches for sparse optimization.
\newblock {\em European Journal of Operational Research}, 244(1):26--46, 2015.

\bibitem{lian2015asynchronous}
X.~Lian, Y.~Huang, Y.~Li, and J.~Liu.
\newblock Asynchronous parallel stochastic gradient for nonconvex optimization.
\newblock In {\em Advances in Neural Information Processing Systems}, pages
  2719--2727, 2015.

\bibitem{Liu1995}
J.~Liu.
\newblock Sensitivity analysis in nonlinear programs and variational
  inequalities via continuous selections.
\newblock {\em SIAM Journal on Control and Optimization}, 33:1040--1060, 1955.

\bibitem{liu2015asyspcd}
J.~Liu and S.~J. Wright.
\newblock Asynchronous stochastic coordinate descent: Parallelism and
  convergence properties.
\newblock {\em SIAM Journal on Optimization}, 25(1):351--376, 2015.

\bibitem{liu2015asynchronous}
J.~Liu, S.~J. Wright, C.~R{\'e}, V.~Bittorf, and S.~Sridhar.
\newblock An asynchronous parallel stochastic coordinate descent algorithm.
\newblock {\em The Journal of Machine Learning Research}, 16(1):285--322, 2015.

\bibitem{Mania_et_al_stochastic_asy2016}
H.~Mania, X.~Pan, D.~Papailiopoulos, B.~Recht, K.~Ramchandran, and M.~I.
  Jordan.
\newblock Perturbed iterate analysis for asynchronous stochastic optimization.
\newblock {\em arXiv:1507.06970}, 2016.

\bibitem{nesterov2012gradient}
Y.~Nesterov.
\newblock {Gradient methods for minimizing composite functions}.
\newblock {\em Mathematical Programming}, 140:125--161, August 2013.

\bibitem{Hogwild!}
F.~Niu, B.~Recht, C.~Re, and S.~J. Wright.
\newblock Hogwild: a lock-free approach to parallelizing stochastic gradient
  descent.
\newblock {\em Advances in Neural Information Processing Systems}, pages
  693--701, 2011.

\bibitem{peng2015arock}
Z.~Peng, Y.~Xu, M.~Yan, and W.~Yin.
\newblock Arock: an algorithmic framework for asynchronous parallel coordinate
  updates.
\newblock {\em arXiv preprint arXiv:1506.02396}, 2015.

\bibitem{RalphDempe1995}
D.~Ralph and S.~Dempe.
\newblock Directional derivatives of the solution of a parametric nonlinear
  program.
\newblock {\em Mathematical programming}, 70(1-3):159--172, 1995.

\bibitem{RFLEXA}
M.~Razaviyayn, M.~Hong, Z.~Q. Luo, and J.~S. Pang.
\newblock Parallel successive convex approximation for nonsmooth nonconvex
  optimization.
\newblock {\em Advances in Neural Information Processing Systems}, pages
  1440--1448, 2014.

\bibitem{scutari2014distributed}
G.~Scutari, F.~Facchinei, and L.~Lampariello.
\newblock Parallel and distributed methods for constrained nonconvex
  optimization-part i: Theory.
\newblock {\em IEEE Transactions on Signal Processing, published online DOI:
  10.1109/TSP.2016.2637317}, 2016.

\bibitem{scutari2016parallel}
G.~Scutari, F.~Facchinei, L.~Lampariello, S.~Sardellitti, and P.~Song.
\newblock Parallel and distributed methods for constrained nonconvex
  optimization--part ii: Applications in communications and machine learning.
\newblock {\em IEEE Transactions on Signal Processing, published online DOI:
  10.1109/TSP.2016.2637314}, 2016.

\bibitem{young1912classes}
W.~H. Young.
\newblock On classes of summable functions and their fourier series.
\newblock {\em Proceedings of the Royal Society of London. Series A, Containing
  Papers of a Mathematical and Physical Character}, 87(594):225--229, 1912.

\bibitem{zhang2014minimization}
S.~Zhang and J.~Xin.
\newblock Minimization of transformed l\_1 penalty: Theory, difference of
  convex function algorithm, and robust application in compressed sensing.
\newblock {\em arXiv preprint arXiv:1411.5735}, 2014.

\end{thebibliography}

		\section{Appendix}
		\subsection{Preliminaries} \label{subsec:preliminary_results}
		We first introduce some preliminary definitions and results that will be instrumental to prove Theorem \eqref{compl}. Some of the following results have already been discussed in \cite{companion2}, to which we refer the interested reader for more details.
		
		In the rest of the Appendix it will be convenient to use the following notation for the random variables and their realizations:  underlined symbols denote random variables, e.g.,  $\underline{\mathbf{x}}^k$, $\underline{\tilde{\mathbf{x}}}^k$ whereas the same symbols with no underline are the corresponding  realizations, e.g.,  $\mathbf{x}^k\triangleq\underline{\mathbf{x}}^k(\omega)$ and $\tilde{\mathbf{x}}^k\triangleq\underline{\tilde{\mathbf{x}}}^k(\omega)$.
		\begin{asparaenum}		
		\item[\bf 1.  Properties of the best response $\hat{\mathbf{x}}(\bullet)$ and Assumption E.] 
		
		\noindent The following proposition is a direct consequence of  \cite[Lemma 7]{scutari2014distributed}.		%We introduce next some basic properties of the best-response  maps defined in \eqref{ncc_2}.
		
		\begin{proposition}\label{Prop_best_response_ncc}
			Given $\hat{\mathbf{x}}(\bullet)$ as defined in \eqref{ncc_2}, under Assumptions A-C  the following holds.
		For any $i \in \mathcal{N}$ and $\mathbf{y} \in \mathcal{K}$,
				\begin{equation}
				(\hat{\mathbf{x}}_i(\mathbf{y}) - \mathbf{y}_i)^T \nabla_{\mathbf{y}_i} f(\mathbf{y}) + g_i(\hat{\mathbf{x}}_i(\mathbf{y})) - g_i(\mathbf{y}_i) \leq - c_{\tilde{f}}\|\hat{\mathbf{x}}_i(\mathbf{y}) - \mathbf{y}_i\|_2^2 \,.
				\end{equation}
		\end{proposition}

The  discussion of when Assumption E is satisfied, i.e. of when the Lipschitz continuity of 	$\hat{\mathbf{x}}(\bullet)$ holds,  is in general complex. Fortunately,  for the kind  of problems we are interested in, it simplifies considerably even if a case by case analysis may be needed.  The following proposition shows in the two fundamental cases in which we are interested  that Assumption E either holds automatically or can be guaranteed by suitable constraint qualifications.

\begin{proposition}\label{Prop_best_response_ncc E}
Suppose that Assumptions A-C hold and that $\cal X$ is compact. Then the following two assertions hold.
\begin{description}
\item[\rm (a) ]  If the feasible set of problem  \eqref{ncc_1} is convex (and therefore there are no nonconvex constraints $c$), then 
Assumption E is satisfied.
\item[\rm (b)] Suppose, for simplicity of presentation only, that the sets ${\cal X}_i$ are specified by a finite set of
convex constraints $h_i(\mathbf{x}_i)\leq 0$, with each component of $h_i$ convex and continuously differentiable. Consider problem \eqref{ncc_1}
in the most common case in which $G(\mathbf{x}) = \lambda \| \mathbf{x} \|_1$ for some positive constant $\lambda$. Assume that $\tilde f$ and $\tilde g$ are $C^2$
and    that, for each $\mathbf{y} \in {\cal X}$, problem \eqref{ncc_2} satisfies the Mangasarian-Fromovitz constraints qualification (MFCQ) and the Constant Rank constraint qualification (CRCQ) in $\mathbf{y}$. Then Assumption E is satisfied.
\end{description}
\end{proposition}
\begin{proof}
Case (a) is nothing else but \cite[Proposition 8 (a)]{FLEXA}. Thus, we focus next on  case (b). Because of the compactness of ${\cal X}_i
$ it is enough to show that every $\hat{\mathbf{x}}_i(\bullet)$ is locally Lipschitz. $\hat{\mathbf{x}}_i(\bullet)$  is the 
unique solution of the strongly convex problem \eqref{ncc_2}. This problem can be equivalently rewritten, by adding extra 
variables $ \mathbf{t}\triangleq (t_1, t_2, \ldots, t_{n_i})$,  as\vspace{-0.2cm}
\begin{equation}\label{ncc_2 epi}
\begin{array}{cl}
\underset{\mathbf{x}_i, \mathbf{t}}{\min}
 &
 \tilde f_{i}(\bx_i; {\bx}^k) + \lambda\,(t_1 + \cdots + t_{n_i})\\[0.5em]
   \text{subject to}	 &  \tilde{c}_{j_i}(\mathbf{x}_i;\mathbf{x}^k_i)\leq 0,\quad j_i=m_{i-1}+1,\ldots,m_i\\[0.5em]
	& h_i(\mathbf{x}_i)\leq 0\\[0.5em]
     & -t_\ell \leq (\mathbf{x}_i)_\ell \leq t_\ell, \quad \ell = 1, \ldots, n_i.
\end{array}
\end{equation}

Since problem \eqref{ncc_2} has the unique solution $\hat{\mathbf{x}}_i( {\bx}^k)$, also 
 problem \eqref{ncc_2 epi} has the unique solution  $(\hat{\mathbf{x}}_i( {\bx}^k),\hat t_1, \ldots, \hat t_{n_i}) $, with
 $\hat t_\ell \triangleq | [\hat{\mathbf{x}}_i( {\bx}^k)]_\ell |$.
 Note also that, by the particular structure of the new linear constraints added in \eqref{ncc_2 epi} with respect to \eqref{ncc_2}, and by the assumptions made on problem \eqref{ncc_2}, problem \eqref{ncc_2 epi} satisfies the MFCQ and the CRCQ at its solution. Finally, observe that assumption B1 ensures that the Lagrangian of problem \eqref{ncc_2 epi}
 is positive definite at the solution of this problem. Then, it is easy to see that the Theorem holds by, for example,
 \cite[Theorem 2]{RalphDempe1995} or \cite[Theorem 3.6]{Liu1995}.
  \end{proof}

\begin{remark}\rm
We note that the assumption that $\cal X$ be  compact can always be satisfied in our setting by suitably redefining $\cal X$, if needed. In fact, A5 guarantees that $\cal K$ is compact,and therefore if $\cal X$ is not compact we can simply redefine it by intersecting it with a suitably large ball without changing problem \eqref{ncc_1}
\end{remark}

\begin{remark}\rm
Note that the line of proof used in Proposition \ref{Prop_best_response_ncc E} (b) can be adapted to deal with all cases in 
which the ``epigraphical transformation'' of the problem leads to smooth constraints and   does not destroys the MFCQ 
and CRCQ of the original problem. For example we can cover in this way the case in which the function $G$ represents  a group $\ell_2$ or $\ell_\infty$ regularizations.
\end{remark}		
		%%%%%%%%%%%%%%%%%%%%%%%%
%				\item [{(b) [Lipschitz continuity]:}] If, in addition, $\mathcal{K}$ is compact, for any $i \in \mathcal{N}$ and $
%\mathbf{y}, \mathbf{z}\in \mathcal{K}$ (we need more assumptions here, TBD):
%				\begin{equation}
%				\|\hat{\mathbf{x}}_i(\mathbf{y}) - \hat{\mathbf{x}}_i(\mathbf{z})\|_2 \leq L_{\hat{x}}\|\mathbf{y}-\mathbf{z}\|_2,
%				\end{equation}
%				with $L_{\hat{x}}>0$.
	%%%%%%%%%%%%%%%%%%%%

		\item[\bf 2. Young's Inequality \cite{young1912classes}.] Consider the Young's inequality in the following form:
		\begin{equation}
		\mu_1\mu_2 \leq \frac{1}{2}(\alpha\mu_1^2 + \alpha^{-1}\mu_2^2) \, ,
		\label{eq:prel_3}
		\end{equation}
		for any $\alpha, \mu_1, \mu_2 > 0$.
		
		\item[\bf 3. Further definitions.]  In order to define a $\sigma$-algebra on $\Omega$ we consider, for
		every  $k\geq 0$ and every $\boldsymbol{\omega}^{0:k}\in {\cal N}\times {\cal D}$,
		the  cylinder
		$$C^k(\boldsymbol{\omega}^{0:k}) \triangleq \{\omega\in \Omega: \boldsymbol{\omega}_{0:k} =
		\boldsymbol{\omega}^{0:k}
		\},$$ i.e., $C^k(\boldsymbol{\omega}^{0:k})$ is the subset of $\Omega$ of all elements $\omega$ whose first $k$
		elements are equal to $\boldsymbol{\omega}^0, \ldots \boldsymbol{\omega}^k$. With a little abuse of notation, we indicate by $\boldsymbol{\omega}_k$ the $k$-th element of the sequence $\omega\in \Omega$. Let us now
		denote by ${\cal C}^k$ the set of all possible $C^k(\boldsymbol{\omega}^{0:k})$  when $\boldsymbol{\omega}^t$, $t=0,
		\ldots, k$,  takes all possible values. Denoting by $\sigma\left({\cal C}^k\right)$ the sigma-algebra generated by  ${\cal C}^k$, define for all $k$,
		\begin{equation}\label{eq:sub-sigma}
		{\cal F}^k\triangleq \sigma\left({\cal C}^k\right) \qquad \mbox{\rm and}\qquad
		{\cal F}\triangleq \sigma\left(\cup_{t=0}^\infty{\cal C}^t\right).
		\end{equation}
		We have ${\cal F}^k\subseteq {\cal F}^{k+1} \subseteq {\cal F}$ for all $k$. The latter inclusion is obvious, the former
		derives easily from the fact that any cylinder in ${\cal C}^{k-1}$ can be obtained as a finite union of cylinders in ${\cal C}^k$.
		\\
		
		Finally, let us define the vectors $\mathbf{w}_\mathbf{x}^k$ such $[\mathbf{w}_\mathbf{x}^k]_i=[\hat{\mathbf{x}}
		(\tilde{\mathbf{x}}^k(\mathbf{d}_{j_{i,k}}))]_i$ where $j_{i,k}$ is defined to be
		$j_{i,k} =\argmax_{j:\,(i,\mathbf{d}_j)\in \mathcal{V}(\omega)} \|\hat{\mathbf{x}}_i
		(\tilde{\mathbf{x}}^k(\mathbf{d}_{j}))-\mathbf{x}^k_i\|_2$. If the $\argmax$ is not a singleton, we just pick the
		first index among those satisfying the operation.

%\begin{definition}
	%The set  $\mathcal{V}(\omega)$ is said to \textit{cover} $\mathcal{N}$ if  $(i,\mathbf{d})\in\mathcal{V}(\omega)$, for some $\mathbf{d}\in\mathcal{D}$ and all $i\in\mathcal{N}$.\end{definition}
%Note that under (C2) $\mathcal{V}(\omega)$ covers $\mathcal{N}$ and consequently .

	%	Given $\omega\in\bar{\Omega}$, let us enumerate the pairs $(i,\mathbf{d})$ in $\mathcal{V}(\omega)$ from $1$ to $|\mathcal{V}(\omega)|$ sorting the pairs in increasing order with respect to their $i$-components and, for a given $i$-component, arbitrarily with respect to the $d$-component. 
\item[\bf 4. Inconsistent read.]
For any given $\omega\in\Omega$, recall that for simplicity of notation  we  define: $\tilde{\mathbf{x}}^k=\mathbf{x}^{k-\mathbf{d}^k}$. Since at each iteration only one block of variables is updated, it is not difficult to see that $\tilde{\mathbf{x}}^k$ can be written as
\begin{equation}
\tilde{\mathbf{x}}^k=\mathbf{x}^k+\sum\limits_{l\in K(\mathbf{d}^k)}(\mathbf{x}^l-\mathbf{x}^{l+1}),
\label{eq:x_tilde}
\end{equation}
where $K(\mathbf{d}^k)\subseteq\{k-\delta,\ldots,k-1\}$ [cf. Assumption D1].
When we need to explicitly specify the dependence of    $\tilde{\mathbf{x}}^k$  on a given realization $\mathbf{d}$ of $\underline{\mathbf{d}}^k$, we will write  $\tilde{\mathbf{x}}^k(\mathbf{d})$.

		\item[\bf 5. Lemma.]
		In order to prove the complexity result we will use the following lemma.
		\begin{lemma}\label{lemma_wright}
			Set $\gamma\leq \frac{(1-\rho^{-1})}{2(1+L_{\hat{x}}N(3+2\psi))}$. Given $\mathbf{x}^0$ and $\omega\in\bar{\Omega}$, the sequence generated by the proposed algorithm under all the previous assumptions satisfies the following condition for any $k\geq0$:
			\begin{equation}
			\|\mathbf{w}_\mathbf{x}^{k-1}-\mathbf{x}^{k-1}\|_2^2\leq\rho\|\mathbf{w}_\mathbf{x}^k-\mathbf{x}^k\|_2^2.
			\end{equation}
		\end{lemma}
		\begin{proof}
			The proof is done by induction and parallels, to some extent, a similar one  presented in \cite{liu2015asyspcd}.\\
			Let us consider a given $\omega\in\bar{\Omega}$ and $k\geq0$.\\
			We start by relying on the following trivial implication, which holds true for any two vectors $\mathbf{a}, \mathbf{b} \in \mathbb{R}^n$:
			\begin{equation}
			\|\mathbf{a} - \mathbf{b}\|_2^2 = \|\mathbf{a}\|_2^2 + \|\mathbf{b}\|_2^2 - 2\mathbf{a}^T\mathbf{b} \geq 0 \Longrightarrow \|\mathbf{a}\|_2^2 + \|\mathbf{b}\|_2^2 \geq 2\mathbf{a}^T\mathbf{b},
			\end{equation}
			that, combined with Cauchy-Schwartz inequality, leads to: 
\begin{align}
	&\nonumber		\|\mathbf{a}\|_2^2-\|\mathbf{b}\|_2^2 	
			= 2\|\mathbf{a}\|_2^2-(\|\mathbf{a}\|_2^2+\|\mathbf{b}\|_2^2)	
			\leq 2\|\mathbf{a}\|_2^2-2\mathbf{a}^\text{T}\mathbf{b} \\	
		&	=2\mathbf{a}^\text{T}(\mathbf{a}-\mathbf{b}) 	
			\leq 2\|\mathbf{a}\|_2\|\mathbf{b}-\mathbf{a}\|_2.
			\label{eq:l01}
			\end{align}
			The following holds:
			\begin{equation}
			\def\arraystretch{2}\begin{array}{l}
			
			\|\gamma(\mathbf{w}_\mathbf{x}^{k-1}-\mathbf{x}^{k-1})\|_2^2-\|\gamma(\mathbf{w}_\mathbf{x}^k - \mathbf{x}^k)\|_2^2\\
			
			\stackrel{\text{(a)}}{\leq} 2\|\gamma(\mathbf{w}_\mathbf{x}^{k-1}-\mathbf{x}^{k-1})\|_2\|\gamma(\mathbf{w}_\mathbf{x}^k - \mathbf{x}^k)-\gamma(\mathbf{w}_\mathbf{x}^{k-1}-\mathbf{x}^{k-1})\|_2\\
			
			\stackrel{\text{(b)}}{\leq}2\|\gamma(\mathbf{w}_\mathbf{x}^{k-1}-\mathbf{x}^{k-1})\|_2(\gamma\|\mathbf{x}^k-\mathbf{x}^{k-1}\|_2+\gamma L_{\hat{x}}\sum\limits_{i=1}^{N}\|\tilde{\mathbf{x}}^k(\mathbf{d}_{j_{i,k}})-\tilde{\mathbf{x}}^{k-1}(\mathbf{d}_{j_{i,k-1}})\|_2)\\
			
			\leq2\|\gamma(\mathbf{w}_\mathbf{x}^{k-1}-\mathbf{x}^{k-1})\|_2(\gamma\|\mathbf{x}^k-\mathbf{x}^{k-1}\|_2+\gamma L_{\hat{x}}N\|\mathbf{x}^k-\mathbf{x}^{k-1}\|_2\\+\gamma L_{\hat{x}}\sum\limits_{i=1}^{N}(\|\mathbf{x}^k-\tilde{\mathbf{x}}^k(\mathbf{d}_{j_{i,k}})\|_2+\|\mathbf{x}^{k-1}-\tilde{\mathbf{x}}^{k-1}(\mathbf{d}_{j_{i,k-1}})\|_2))
			\end{array}
			\label{eq:l02}
			\end{equation}
			where in (a) we used \eqref{eq:l01} and (b) comes from Assumption E.
			We will prove the Lemma by induction, so let us analyze what happens for $k=1$. \eqref{eq:l02} simply becomes:
			\begin{align}
	&\nonumber
			\|\gamma(\mathbf{w}_\mathbf{x}^0-\mathbf{x}^0)\|_2^2-\|\gamma(\mathbf{w}_\mathbf{x}^1-\mathbf{x}^1)\|_2^2
			\leq2\|\gamma(\mathbf{w}_\mathbf{x}^0-\mathbf{x}^0)\|_2((1+L_{\hat{x}}N)\gamma\|\mathbf{x}^1-\mathbf{x}^0\|_2\\
			&+\gamma L_{\hat{x}}\sum\limits_{i=1}^{N}(\|\mathbf{x}^1-\tilde{\mathbf{x}}^1(\mathbf{d}_{j_{i,1}})\|_2+\|\mathbf{x}^0-\tilde{\mathbf{x}}^0(\mathbf{d}_{j_{i,0}})\|_2))
			\label{ncc_new1}		
			\end{align}
			For $k=1$ and for any $j=1,\ldots,|\mathcal{V}(\omega)|$, we can bound the terms in \eqref{ncc_new1} as:
			\begin{align}
			&\nonumber 
			\|\mathbf{x}^1-\tilde{\mathbf{x}}^1(\mathbf{d}_{j_{i,1}})\|_2+\|\mathbf{x}^0-\tilde{\mathbf{x}}^0(\mathbf{d}_{j_{i,0}})\|_2
			\stackrel{\text{(a)}}{=}\|\mathbf{x}^1-\tilde{\mathbf{x}}^1(\mathbf{d}_{j_{i,1}})\|_2+\|\mathbf{x}^0-\mathbf{x}^0\|_2
\\
			&			\leq \|\mathbf{x}^1-\mathbf{x}^0\|_2,
			\label{eq:a23bis}
			\end{align}	
			where (a) comes from \eqref{eq:x_tilde} and the fact that $K(\mathbf{d}^0)=\emptyset$ for any $\mathbf{d}^0\in\mathcal{D}$ and $K(\mathbf{d}^1)\subseteq\{0\}$ for any $\mathbf{d}^1\in\mathcal{D}$ (see the definition of the set $K(\mathbf{d}^k)$ in Section \ref{subsec:preliminary_results}).\\
			Substituting \eqref{eq:a23bis} in \eqref{ncc_new1}, we have:
			\begin{align}
&\nonumber		
			\|\gamma(\mathbf{w}_\mathbf{x}^0 - \mathbf{x}^0)\|_2^2-\|\gamma(\mathbf{w}_\mathbf{x}^1-\mathbf{x}^1)\|_2^2
			\;\leq\; 2\gamma(1+L_{\hat{x}}N)\|\gamma(\mathbf{w}_\mathbf{x}^0-\mathbf{x}^0)\|_2\|\mathbf{x}^1-\mathbf{x}^0\|_2\\
			&\nonumber	+2\gamma L_{\hat{x}}N\|\gamma(\mathbf{w}_\mathbf{x}^0-\mathbf{x}^0)\|_2\|\mathbf{x}^1-\mathbf{x}^0\|_2
					\;=\; 2\gamma(1+2L_{\hat{x}}N)\|\gamma(\mathbf{w}_\mathbf{x}^0-\mathbf{x}^0)\|_2\|\mathbf{x}^1-\mathbf{x}^0\|_2\\			
	&		\label{eq:base_1}
			\stackrel{(a)}{\leq}\gamma(1+2L_{\hat{x}}N)(\|\gamma(\mathbf{w}_\mathbf{x}^0-\mathbf{x}^0)\|_2^2+\|\mathbf{x}^1-\mathbf{x}^0\|_2^2)\\			
	&\nonumber		=\gamma(1+2L_{\hat{x}}N)(\|\gamma(\mathbf{w}_\mathbf{x}^0-\mathbf{x}^0)\|_2^2+\|\gamma(\hat{\mathbf{x}}_{i^0}(\tilde{\mathbf{x}}^0(\mathbf{d}^0))-\mathbf{x}_{i^0}^0)\|_2^2)\\			
	&\nonumber		\leq 2\gamma(1+2L_{\hat{x}}N)\|\gamma(\mathbf{w}_\mathbf{x}^0 - \mathbf{x}^0)\|_2^2,
			\end{align}
			where (a) follows from the Young's inequality with $\alpha=1$ and the last inequality follows from the
definition of $\mathbf{w}_\mathbf{x}$  in Section \ref{subsec:preliminary_results}. We can now derive the base of the induction  we were seeking, just rearranging the terms in \eqref{eq:base_1}:
			\begin{align}
			\|\gamma(\mathbf{w}_\mathbf{x}^0 - \mathbf{x}^0)\|_2^2
			\leq  (1-2\gamma(1+2L_{\hat{x}}N) )^{-1}\|\gamma(\mathbf{w}_\mathbf{x}^1-\mathbf{x}^1)\|_2^2
			\leq \rho\|\gamma(\mathbf{w}_\mathbf{x}^1-\mathbf{x}^1)\|_2^2.
			\label{eq:base_2}
			\end{align}
			The steps in \eqref{eq:base_2} are valid only if, for some $\rho > 1$:
			\begin{equation}
			\gamma\leq \frac{(1-\rho^{-1})}{2(1+2L_{\hat{x}}N)}.
			\label{eq:step_cond1}
			\end{equation}
			In this way we proved the base of the induction. Now we start again from \eqref{eq:l02} in order to finish the proof of the Lemma. Let us search a bound for the following quantity for any $j=1,\ldots,|\mathcal{V}(\omega)|$:
			\begin{equation}
			\def\arraystretch{2}\begin{array}{l}
			\|\mathbf{x}^k-\tilde{\mathbf{x}}^k(\mathbf{d}_{j_{i,k}})\|_2+\|\mathbf{x}^{k-1}-\tilde{\mathbf{x}}^{k-1}(\mathbf{d}_{j_{i,k-1}})\|_2\\
			
			\stackrel{\text{(a)}}{=}\|\sum\limits_{l\in K^k(\mathbf{d}_j)}(\mathbf{x}^{l+1}-\mathbf{x}^l)\|_2+\|\sum\limits_{l\in K^{k-1}(\mathbf{d}_j)}(\mathbf{x}^{l+1}-\mathbf{x}^l)\|_2 \\
			
			\leq\sum\limits_{l\in K^k(\mathbf{d}_{j_{i,k}})}\|\mathbf{x}^{l+1}-\mathbf{x}^l\|_2+\sum\limits_{l\in K^{k-1}(\mathbf{d}_{j_{i,k-1}})}\|\mathbf{x}^{l+1}-\mathbf{x}^l\|_2 \\
			
			\stackrel{\text{(b)}}{\leq}\sum\limits_{l=k-\delta}^{k-1}\|\mathbf{x}^{l+1}-\mathbf{x}^l\|_2+\sum\limits_{l=k-1-\delta}^{k-2}\|\mathbf{x}^{l+1}-\mathbf{x}^l\|_2\\
			
			\leq 2\sum\limits_{l=k-1-\delta}^{k-1}\|\mathbf{x}^{l+1}-\mathbf{x}^l\|_2\;
			=\; 2\sum\limits_{l=k-1-\delta}^{k-1}\|\mathbf{x}^{l+1}_{i^l}-\mathbf{x}^l_{i^l}\|_2
			\end{array}
			\label{eq:a23}
			\end{equation}
			where (a) comes from \eqref{eq:x_tilde} and (b) from D1. Plugging this last result into \eqref{eq:l02}:
			\begin{equation}
			\def\arraystretch{2}\begin{array}{l}
			\|\gamma(\mathbf{w}_\mathbf{x}^{k-1}-\mathbf{x}^{k-1})\|_2^2-\|\gamma(\mathbf{w}_\mathbf{x}^k - \mathbf{x}^k)\|_2^2\\
			
			\leq 2\gamma(1+L_{\hat{x}}N)\|\gamma(\mathbf{w}_\mathbf{x}^{k-1}-\mathbf{x}^{k-1})\|_2\|\mathbf{x}^k-\mathbf{x}^{k-1}\|_2\\+4\gamma L_{\hat{x}}N\|\gamma(\mathbf{w}_\mathbf{x}^{k-1}-\mathbf{x}^{k-1})\|_2\sum\limits_{l=k-1-\delta}^{k-1}\|\mathbf{x}^{l+1}_{i^l}-\mathbf{x}^l_{i^l}\|_2\\
			
			= 2\gamma(1+3L_{\hat{x}}N)\|\gamma(\mathbf{w}_\mathbf{x}^{k-1}-\mathbf{x}^{k-1})\|_2\|\mathbf{x}^k-\mathbf{x}^{k-1}\|_2\\+4\gamma  L_{\hat{x}}N\|\gamma(\mathbf{w}_\mathbf{x}^{k-1}-\mathbf{x}^{k-1})\|_2\sum\limits_{l=k-1-\delta}^{k-2}\|\mathbf{x}^{l+1}_{i^l}-\mathbf{x}^l_{i^l}\|_2.
			\end{array}
			\label{eq:a24}
			\end{equation}
			By using Young's inequality with $\alpha_l>0$ for any $l$, we get:
			\begin{align}
		&\nonumber
			\|\mathbf{x}^{l+1}_{i^l} - \mathbf{x}^l_{i^l}\|_2\|\gamma(\mathbf{w}_\mathbf{x}^{k-1}-\mathbf{x}^{k-1})\|_2
			\;\leq \;\frac{1}{2}(\alpha_l\|\mathbf{x}^{l+1}_{i^l} - \mathbf{x}^l_{i^l}\|_2^2 + \alpha_l^{-1}\|\gamma(\mathbf{w}_\mathbf{x}^{k-1}-\mathbf{x}^{k-1})\|_2^2)\\			
&		\label{eq:l26b}		=\frac{1}{2}(\alpha_l\|\gamma(\hat{\mathbf{x}}_{i^l}(\tilde{\mathbf{x}}^l(\mathbf{d}^l)) - \mathbf{x}_{i^l}^l)\|_2^2 + \alpha_l^{-1}\|\gamma(\mathbf{w}_\mathbf{x}^{k-1}-\mathbf{x}^{k-1})\|_2^2)\\		
&\nonumber			\leq\frac{1}{2}(\alpha_l\|\gamma(\mathbf{w}_\mathbf{x}^l - \mathbf{x}^l)\|_2^2 + \alpha_l^{-1}\|\gamma(\mathbf{w}_\mathbf{x}^{k-1}-\mathbf{x}^{k-1})\|_2^2)		
			\end{align}
			Substituting \eqref{eq:l26b} in \eqref{eq:a24} and using again Young's inequality with $\alpha=1$, we get the following result:
			\begin{align}
	&\nonumber
			\|\gamma(\mathbf{w}_\mathbf{x}^{k-1}-\mathbf{x}^{k-1})\|_2^2-\|\gamma(\mathbf{w}_\mathbf{x}^k-\mathbf{x}^k)\|_2^2\\		
&\nonumber			\leq \gamma(1+3L_{\hat{x}}N)(\|\gamma(\mathbf{w}_\mathbf{x}^{k-1}-\mathbf{x}^{k-1})\|_2^2 +\|\mathbf{x}^k-\mathbf{x}^{k-1}\|_2^2)\\		
&\label{eq:l27b}		+2\gamma N L_{\hat{x}}\sum\limits_{l=k-1-\delta}^{k-2} (\alpha_l\|\gamma(\mathbf{w}_\mathbf{x}^l-\mathbf{x}^l)\|_2^2 + \alpha_l^{-1}\|\gamma(\mathbf{w}_\mathbf{x}^{k-1}-\mathbf{x}^{k-1})\|_2^2)\\		
&\nonumber				=\gamma(1+3L_{\hat{x}}N)(\|\gamma(\mathbf{w}_\mathbf{x}^{k-1}-\mathbf{x}^{k-1})\|_2^2 			+\|\gamma(\hat{\mathbf{x}}_{i^{k-1}}(\tilde{\mathbf{x}}^{k-1}(\mathbf{d}^{k-1}))-\mathbf{x}^{k-1}_{i^{k-1}})\|_2^2)\\
			&\nonumber			+2\gamma N L_{\hat{x}}\sum\limits_{l=k-1-\delta}^{k-2} (\alpha_l\|\gamma(\mathbf{w}_\mathbf{x}^l-\mathbf{x}^l)\|_2^2 + \alpha_l^{-1}\|\gamma(\mathbf{w}_\mathbf{x}^{k-1}-\mathbf{x}^{k-1})\|_2^2)\\
			&\nonumber			\leq2\gamma(1+3L_{\hat{x}}N)\|\gamma(\mathbf{w}_\mathbf{x}^{k-1}-\mathbf{x}^{k-1})\|_2^2\\
			&\nonumber			+2\gamma NL_{\hat{x}}\sum\limits_{l=k-1-\delta}^{k-2} (\alpha_l\|\gamma(\mathbf{w}_\mathbf{x}^l-\mathbf{x}^l)\|_2^2 + \alpha_l^{-1}\|\gamma(\mathbf{w}_\mathbf{x}^{k-1}-\mathbf{x}^{k-1})\|_2^2)		
			\end{align}
			
			Assuming the inductive step to hold true up to the step $\|\mathbf{w}_\mathbf{x}^{k-2}-\mathbf{x}^{k-2}\|_2^2 \leq \rho \|\mathbf{w}_\mathbf{x}^{k-1}-\mathbf{x}^{k-1}\|_2^2$, we obtain:
			\begin{align}
&	\label{eq:l27c}			\|\gamma(\mathbf{w}_\mathbf{x}^{k-1}-\mathbf{x}^{k-1})\|_2^2-\|\gamma(\mathbf{w}_\mathbf{x}^k-\mathbf{x}^k)\|_2^2\;
			\leq \;2\gamma(1+3L_{\hat{x}}N)\|\gamma(\mathbf{w}_\mathbf{x}^{k-1}-\mathbf{x}^{k-1})\|_2^2\\
			&\nonumber +2\gamma NL_{\hat{x}}\sum\limits_{l=k-\delta-1}^{k-2}( \alpha_l\rho^{k-l-1}\|\gamma(\mathbf{w}_\mathbf{x}^{k-1}-\mathbf{x}^{k-1})\|_2^2+ \alpha_l^{-1}\|\gamma(\mathbf{w}_\mathbf{x}^{k-1}-\mathbf{x}^{k-1})\|_2^2).
			\end{align}
			Setting $\alpha_l=(\rho^{\frac{1+l-k}{2}})^{-1}$ and noticing that $\sum\limits_{l=k-\delta-1}^{k-2} \rho^{\frac{k-l-1}{2}}=\sum\limits_{t=1}^\delta\rho^{\frac{t}{2}}=\psi$ we get:
			\begin{align}
&\nonumber
			\|\gamma(\mathbf{w}_\mathbf{x}^{k-1}-\mathbf{x}^{k-1})\|_2^2-\|\gamma(\mathbf{w}_\mathbf{x}^k-\mathbf{x}^k)\|_2^2
			\leq 2\gamma(1+3L_{\hat{x}}N)\|\gamma(\mathbf{w}_\mathbf{x}^{k-1}-\mathbf{x}^{k-1})\|_2^2\\
	&		\label{eq:l27d}	+4\gamma NL_{\hat{x}}\sum\limits_{l=k-\delta-1}^{k-2} \rho^{\frac{k-l-1}{2}}\|\gamma(\mathbf{w}_\mathbf{x}^{k-1}-\mathbf{x}^{k-1})\|_2^2\\
	&\nonumber		
			=2\gamma(1+3L_{\hat{x}}N)\|\gamma(\mathbf{w}_\mathbf{x}^{k-1}-\mathbf{x}^{k-1})\|_2^2+4\gamma NL_{\hat{x}}\psi\|\gamma(\mathbf{w}_\mathbf{x}^{k-1}-\mathbf{x}^{k-1})\|_2^2.
			\end{align}
			Rearranging the terms we get the desired result
			\begin{align}
	&\label{eq:l27f}
			\|\gamma(\mathbf{w}_\mathbf{x}^{k-1}-\mathbf{x}^{k-1})\|_2^2\\	
&\nonumber			\leq (1- 2\gamma(1+L_{\hat{x}}N(3+2\psi)))^{-1}\|\gamma(\mathbf{w}_\mathbf{x}^k-\mathbf{x}^k)\|_2^2\;	\leq\; \rho\|\gamma(\mathbf{w}_\mathbf{x}^k-\mathbf{x}^k)\|_2^2.
			\end{align}
			\eqref{eq:l27f} holds true if:
			\begin{equation}
			\gamma\leq \frac{(1-\rho^{-1})}{2(1+L_{\hat{x}}N(3+2\psi))}
			\label{eq:step_cond}
			\end{equation}
			with $\rho>1$. Note that if we set: $\gamma\leq \frac{(1-\rho^{-1})}{2(1+L_{\hat{x}}N(3+2\psi))}$, both condition \eqref{eq:step_cond1} and \eqref{eq:step_cond} are satisfied.
		\end{proof}\medskip
	\end{asparaenum}
		\subsection{Proof of Theorem \ref{compl}}
		
		Let us define:
		\begin{equation}
		\hat{\mathbf{y}}_i^k=\underset{\mathbf{y}_i\in\mathcal{K}_i(\mathbf{x}^k_i)}{\text{arg min}}\{\nabla_{\mathbf{x}_i}f(\mathbf{x}^k)^\text{T}(\mathbf{y}_i-\mathbf{x}_i^k)+g_i(\mathbf{y}_i)+\frac{1}{2}\|\mathbf{y}_i-\mathbf{x}_i^k\|^2_2\}
		\label{34}
		\end{equation}
		and note that $M_F(\mathbf{x})=[\mathbf{x}_1^k-\hat{\mathbf{y}}_1^k,\ldots,\mathbf{x}_N^k-\hat{\mathbf{y}}_N^k]^\text{T}$. Let us consider in the following a given realization $\omega\in\bar{\Omega}$ and $k\geq0$. Relying on the first order optimality conditions for $\hat{\mathbf{y}}^k_{i^k}$ and using convexity of $g_{i^k}$ we can write the following inequality, that holds true for any $\mathbf{z}_{i^k} \in \mathcal{K}_{i^k}(\mathbf{x}_{i^k}^k)$:
		\begin{equation}
		(\nabla_{\mathbf{x}_{i^k}} f(\mathbf{x}^k) + \hat{\mathbf{y}}_{i^k}- \mathbf{x}^k_{i^k})^\text{T}(\mathbf{z}_{i^k} - \hat{\mathbf{y}}_{i^k}) + g_{i^k}(\mathbf{z}_{i^k}) - g_{i^k}(\hat{\mathbf{y}}_{i^k}) \geq 0 .
		\label{eq:t1}
		\end{equation}
		In a similar way we can use first order optimality condition for $\hat{\mathbf{x}}_{i^k}(\tilde{\mathbf{x}}^k)$ and convexity of $g_{i^k}$, obtaining:
		\begin{equation}
		\nabla \tilde{f}_{i^k}(\hat{\mathbf{x}}_{i^k}(\tilde{\mathbf{x}}^k);\tilde{\mathbf{x}}^k)^\text{T}(\mathbf{w}_{i^k} - \hat{\mathbf{x}}_{i^k}(\tilde{\mathbf{x}}^k)) + g_{i^k}(\mathbf{w}_{i^k}) - g_{i^k}(\hat{\mathbf{x}}_{i^k}(\tilde{\mathbf{x}}^k))) \geq 0 \, ,
		\label{eq:t2}
		\end{equation}
		that holds true for any $\mathbf{w}_i \in \mathcal{K}_{i^k}(\tilde{\mathbf{x}}_i^k)=\mathcal{K}_{i^k}(\mathbf{x}_i^k)$ (cf. D4). We can now sum up \eqref{eq:t1} and \eqref{eq:t2} together setting $\mathbf{w}_{i^k} = \hat{\mathbf{y}}_{i^k}$ and $\mathbf{z}_{i^k} = \hat{\mathbf{x}}_{i^k}(\tilde{\mathbf{x}}^k)$:
		\begin{equation}
		(\nabla_{\mathbf{x}_{i^k}} f(\mathbf{x}^k) - \nabla \tilde{f}_{i^k}(\hat{\mathbf{x}}_{i^k}(\tilde{\mathbf{x}}^k);\tilde{\mathbf{x}}^k) + \hat{\mathbf{y}}_{i^k} - \mathbf{x}^k_{i^k} )^\text{T}(\hat{\mathbf{x}}_{i^k}(\tilde{\mathbf{x}}^k) - \hat{\mathbf{y}}_{i^k})  \geq 0.
		\label{eq:t3}
		\end{equation}
		Summing and subtracting $\hat{\mathbf{x}}_{i^k}(\tilde{\mathbf{x}}^k)$ inside the first parenthesis on the left hand side and using the gradient consistency assumption B2 we get:
		\begin{equation}
		(\nabla \tilde{f}_{i^k}(\mathbf{x}^k_{i^k};\mathbf{x}^k) - \nabla \tilde{f}_{i^k}(\hat{\mathbf{x}}_{i^k}(\tilde{\mathbf{x}}^k);\tilde{\mathbf{x}}^k) + \hat{\mathbf{x}}_{i^k}(\tilde{\mathbf{x}}^k) - \mathbf{x}^k_{i^k} )^\text{T}(\hat{\mathbf{x}}_{i^k}(\tilde{\mathbf{x}}^k) - \hat{\mathbf{y}}_{i^k})\geq \|\hat{\mathbf{x}}_{i^k}(\tilde{\mathbf{x}}^k) - \hat{\mathbf{y}}_{i^k}\|_2^2 \, .
		\label{eq:t4}
		\end{equation}
		Applying Cauchy-Schwartz inequality to upper-bound the left hand side of \eqref{eq:t4} we obtain:
		\begin{equation}
		\|\nabla \tilde{f}_{i^k}(\mathbf{x}^k_{i^k};\mathbf{x}^k) - \nabla \tilde{f}_{i^k}(\hat{\mathbf{x}}_{i^k}(\tilde{\mathbf{x}}^k);\tilde{\mathbf{x}}^k) + \hat{\mathbf{x}}_{i^k}(\tilde{\mathbf{x}}^k) - \mathbf{x}^k_{i^k}\|_2 \geq \|\hat{\mathbf{x}}_{i^k}(\tilde{\mathbf{x}}^k) - \hat{\mathbf{y}}_{i^k}\|_2 \, ,
		\label{eq:t5}
		\end{equation}
		We can proceed summing and subtracting $\nabla \tilde{f}_{i^k}(\hat{\mathbf{x}}_{i^k}(\tilde{\mathbf{x}}^k);\mathbf{x}^k)$ inside the norm on the left hand side and then applying triangular inequality to get:
		\begin{equation}
		\def\arraystretch{2}\begin{array}{l}
		\|\nabla \tilde{f}_{i^k}(\hat{\mathbf{x}}_{i^k}(\tilde{\mathbf{x}}^k);\mathbf{x}^k) - \nabla \tilde{f}_{i^k}(\hat{\mathbf{x}}_{i^k}(\tilde{\mathbf{x}}^k);\tilde{\mathbf{x}}^k)\|_2\\+ \|\nabla \tilde{f}_{i^k}(\mathbf{x}^k_{i^k};\mathbf{x}^k) - \nabla \tilde{f}_{i^k}(\hat{\mathbf{x}}_{i^k}(\tilde{\mathbf{x}}^k);\mathbf{x}^k)\|_2+ \|\hat{\mathbf{x}}_{i^k}(\tilde{\mathbf{x}}^k) - \mathbf{x}^k_{i^k}\|_2\\
		
		\geq \|\hat{\mathbf{x}}_{i^k}(\tilde{\mathbf{x}}^k) - \hat{\mathbf{y}}_{i^k}\|_2,
		\end{array}
		\label{eq:t6}
		\end{equation}
		and we can further upper-bound the left hand side relying on the Lipschitz continuity assumptions B3 and B4:
		\begin{equation}
		\|\hat{\mathbf{x}}_{i^k}(\tilde{\mathbf{x}}^k) - \hat{\mathbf{y}}_{i^k}\|_2 \leq (1+L_E)\|\hat{\mathbf{x}}_{i^k}(\tilde{\mathbf{x}}^k) - \mathbf{x}^k_{i^k}\|_2 +L_B\| \mathbf{x}^k-\tilde{\mathbf{x}}^k\|_2\, .
		\label{eq:t7}
		\end{equation}
		Taking the square both sides we have:
		\begin{align}
		&\nonumber
		\|\hat{\mathbf{x}}_{i^k}(\tilde{\mathbf{x}}^k) - \hat{\mathbf{y}}_{i^k}\|_2^2\;		
		\leq\; (1+L_E)^2\|\hat{\mathbf{x}}_{i^k}(\tilde{\mathbf{x}}^k) - \mathbf{x}^k_{i^k}\|_2^2 + L_B^2\| \mathbf{x}^k-\tilde{\mathbf{x}}^k \|_2^2\\
		&+ 2L_B(1+L_E)\|\hat{\mathbf{x}}_{i^k}(\tilde{\mathbf{x}}^k) - \mathbf{x}^k_{i^k}\|_2\| \mathbf{x}^k-\tilde{\mathbf{x}}^k\|_2.
		\label{eq:t8}
		\end{align}
		We are now interested in bounding $\|\mathbf{x}_{i^k}^k - \hat{\mathbf{y}}_{i^k}\|_2^2$:
		\begin{align}
&\label{eq:compl1}			\|\mathbf{x}_{i^k}^k - \hat{\mathbf{y}}_{i^k}\|_2^2
		= \|\mathbf{x}_{i^k}^k - \hat{\mathbf{x}}_{i^k}(\tilde{\mathbf{x}}^k) +\hat{\mathbf{x}}_{i^k}(\tilde{\mathbf{x}}^k)- \hat{\mathbf{y}}_{i^k}\|_2^2 \\
	&	\nonumber
		\leq \;2\left(\|\hat{\mathbf{x}}_{i^k}(\tilde{\mathbf{x}}^k) - \mathbf{x}_{i^k}^k\|_2^2 + \|\hat{\mathbf{x}}_{i^k}(\tilde{\mathbf{x}}^k)- \hat{\mathbf{y}}_{i^k}\|_2^2\right) 
	\stackrel{(a)}{\leq} (2+2(1+L_E)^2)\|\hat{\mathbf{x}}_{i^k}(\tilde{\mathbf{x}}^k) - \mathbf{x}_{i^k}^k\|_2^2\\
	&\nonumber		+ 2L_B^2\| \mathbf{x}^k-\tilde{\mathbf{x}}^k \|_2^2 + 4L_B(1+L_E)\|\hat{\mathbf{x}}_{i^k}(\tilde{\mathbf{x}}^k) - \mathbf{x}^k_{i^k}\|_2\| \mathbf{x}^k-\tilde{\mathbf{x}}^k \|_2
		\end{align}
		where (a) comes from \eqref{eq:t8}.
		Taking now conditional expectation and using D2:
		\begin{equation}
		\mathbb{E}(\|\mathbf{x}_{\underline{i}^k}^k - \hat{\mathbf{y}}_{\underline{i}^k}\|_2^2|\mathcal{F}^{k-1})(\omega)
		\;=\; \sum\limits_{i=1}^Np(i|\omega^{0:k-1})\|\mathbf{x}_i^k - \hat{\mathbf{y}}_i\|_2^2\;	
		\geq\; p_\text{min}\|\mathbf{x}^k-\hat{\mathbf{y}}\|_2^2,
		\label{eq:compl2}
		\end{equation}
		It follows:
		\begin{align}
&\nonumber		p_\text{min}\|\mathbf{x}^k - \hat{\mathbf{y}}\|_2^2\\
&\nonumber			
		\stackrel{(a)}{\leq} (2+2(1+L_E)^2)\mathbb{E}(\|\hat{\mathbf{x}}_{\underline{i}^k}(\underline{\tilde{\mathbf{x}}}^k) - \mathbf{x}^k_{\underline{i}^k}\|_2^2|\mathcal{F}^{k-1})(\omega)
			+ 2L_B^2\mathbb{E}(\|\mathbf{x}^k-\underline{\tilde{\mathbf{x}}}^k\|_2^2|\mathcal{F}^{k-1})(\omega)\\
			&\nonumber	+ 4L_B(1+L_E)\mathbb{E}(\|\hat{\mathbf{x}}_{\underline{i}^k}(\underline{\tilde{\mathbf{x}}}^k) - \mathbf{x}^k_{\underline{i}^k}\|_2\|\mathbf{x}^k-\underline{\tilde{\mathbf{x}}}^k\|_2|\mathcal{F}^{k-1})(\omega)\\
		&\nonumber
		\stackrel{(b)}{\leq} 2(1+(1+L_E)(1+L_B+L_E))\|\mathbf{w}_\mathbf{x}^k-\mathbf{x}^k\|_2^2\\
		&\label{eq:t9}+2L_B(1+L_B+L_E)\mathbb{E}(\|\mathbf{x}^k-\underline{\tilde{\mathbf{x}}}^k\|_2^2|\mathcal{F}^{k-1})(\omega),		
		\end{align}
		where in (a) we combined \eqref{eq:compl1} and \eqref{eq:compl2} together while (b) follows from the Young's inequality with $\alpha=1$ and the definition of $\mathbf{w}_\mathbf{x}$.\\
		In order to bound \eqref{eq:t9} we need the following result:
		\begin{align}
&\nonumber
		\mathbb{E}(\|\mathbf{x}^k-\underline{\tilde{\mathbf{x}}}^k\|_2^2|\mathcal{F}^{k-1})(\omega)	
	\;	\stackrel{(a)}{\leq}\;\mathbb{E}\left(\left(\sum\limits_{l=k-\delta}^{k-1}\|\mathbf{x}^{l+1}-\mathbf{x}^l\|_2^2\right)^2|\mathcal{F}^{k-1}\right)(\omega)\\
	&	\label{eq:9gen}
		=\left(\sum\limits_{l=k-\delta}^{k-1}\|\mathbf{x}^{l+1}-\mathbf{x}^l\|_2^2\right)^2
		\;\stackrel{(b)}{\leq}\;\delta\sum\limits_{l=k-\delta}^{k-1}\|\mathbf{x}^{l+1}-\mathbf{x}^l\|_2^2\\		
&\nonumber		=\delta\gamma^2\sum\limits_{l=k-\delta}^{k-1}\|\hat{\mathbf{x}}_{i^l}(\tilde{\mathbf{x}}^l)-\mathbf{x}_{i^l}^l\|_2^2		
	\;	\leq\;\delta\gamma^2\sum\limits_{l=k-\delta}^{k-1}\|\mathbf{w}_\mathbf{x}^l-\mathbf{x}^l\|_2^2		
		\;\stackrel{(c)}{\leq}\;\delta\psi'\gamma^2\|\mathbf{w}_\mathbf{x}^k-\mathbf{x}^k\|_2^2
		\end{align}
		where (a) follows from \eqref{eq:x_tilde}; (b) from the Jensen's inequality; and (c) from Lemma \ref{lemma_wright}.\\
		The left hand side of \eqref{eq:t9} is nothing else than $\|M_F(\mathbf{x}^k)\|_2^2$, which is our optimality measure; substituting \eqref{eq:9gen} into \eqref{eq:t9}:
		\begin{equation}
		p_\text{min}\|M_F(\mathbf{x}^k)\|_2^2
		\leq2(1+(1+L_B+L_E)(1+L_EL_B\delta\psi'\gamma^2))\|\mathbf{w}_\mathbf{x}^k-\mathbf{x}^k\|_2^2.
		\label{eq:42}
		\end{equation}
		We now need a bound for the quantity $\|\mathbf{w}_\mathbf{x}^k-\mathbf{x}^k\|_2^2$. For any $k\geq0$ we have:
		\begin{equation}
		\def\arraystretch{2}\begin{array}{l}
		
		F(\mathbf{x}^{k+1}) \\
		
		= f(\mathbf{x}^{k+1}) + g(\mathbf{x}^{k+1})\\
		
		\stackrel{\text{(a)}}{=} f(\mathbf{x}^{k+1}) + \sum\limits_{i \neq i^k} g_i(\mathbf{x}_i^{k+1}) + g_{i^k}(\mathbf{x}_{i^k}^{k+1})\\
		
		\stackrel{\text{(b)}}{=} f(\mathbf{x}^{k+1}) + \sum\limits_{i \neq i^k} g_i(\mathbf{x}_i^k) + g_{i^k}(\mathbf{x}_{i^k}^{k+1})\\
		
		\stackrel{\text{(c)}}{\leq} f(\mathbf{x}^k)  +\gamma\nabla_{\mathbf{x}_{i^k}} f (\mathbf{x}^k)^\text{T}(\hat{\mathbf{x}}_{i^k}(\tilde{\mathbf{x}}^k) - \mathbf{x}_{i^k}^k)+ \frac{\gamma^2L_f}{2}\|\hat{\mathbf{x}}_{i^k}(\tilde{\mathbf{x}}^k) - \mathbf{x}_{i^k}^k\|_2^2\\+ \sum\limits_{i \neq i^k} g_i(\mathbf{x}_i^k)+ g_{i^k}(\mathbf{x}_{i^k}^{k+1})\\
		
		=f(\mathbf{x}^k) + \gamma\nabla_{\mathbf{x}_{i^k}} f (\tilde{\mathbf{x}}^k)^\text{T}(\hat{\mathbf{x}}_{i^k}(\tilde{\mathbf{x}}^k) - \mathbf{x}_{i^k}^k)\\+(\nabla_{\mathbf{x}_{i^k}} f (\mathbf{x}^k)-\nabla_{\mathbf{x}_{i^k}}f(\tilde{\mathbf{x}}^k))^\text{T}(\gamma(\hat{\mathbf{x}}_{i^k}(\tilde{\mathbf{x}}^k) - \mathbf{x}_{i^k}^k))+\frac{\gamma^2L_f}{2}\|\hat{\mathbf{x}}_{i^k}(\tilde{\mathbf{x}}^k) - \mathbf{x}_{i^k}^k\|_2^2\\+ \sum\limits_{i \neq i^k} g_i(\mathbf{x}_i^k)+ g_{i^k}(\mathbf{x}_{i^k}^{k+1})\\
		
		\stackrel{\text{(d)}}{\leq} f(\mathbf{x}^k)  +\gamma\nabla_{\mathbf{x}_{i^k}} f (\tilde{\mathbf{x}}^k)^\text{T}(\hat{\mathbf{x}}_{i^k}(\tilde{\mathbf{x}}^k) - \tilde{\mathbf{x}}_{i^k}^k)\\+(\nabla_{\mathbf{x}_{i^k}}f(\mathbf{x}^k)-\nabla_{\mathbf{x}_{i^k}}f(\tilde{\mathbf{x}}^k))^\text{T}(\gamma(\hat{\mathbf{x}}_{i^k}(\tilde{\mathbf{x}}^k) - \tilde{\mathbf{x}}_{i^k}^k))+ \frac{\gamma^2L_f}{2}\|\hat{\mathbf{x}}_{i^k}(\tilde{\mathbf{x}}^k) - \tilde{\mathbf{x}}_{i^k}^k\|_2^2\\+ \sum\limits_{i \neq i^k} g_i(\mathbf{x}_i^k) + \gamma g_{i^k}(\hat{\mathbf{x}}_{i^k}(\tilde{\mathbf{x}}^k))+g_{i^k}(\mathbf{x}_{i^k}^k)-\gamma g_{i^k}(\tilde{\mathbf{x}}_{i^k}^k)\\
		
		\stackrel{\text{(e)}}{\leq} F(\mathbf{x}^k) -\gamma(c_{\tilde{f}}-\frac{\gamma L_f}{2})\|\hat{\mathbf{x}}_{i^k}(\tilde{\mathbf{x}}^k)-\tilde{\mathbf{x}}_{i^k}^k\|_2^2+L_f\|\mathbf{x}^k-\tilde{\mathbf{x}}^k\|_2\|\gamma(\hat{\mathbf{x}}_{i^k}(\tilde{\mathbf{x}}^k) - \tilde{\mathbf{x}}_{i^k}^k)\|_2\\
		
		\stackrel{\text{(f)}}{\leq} F(\mathbf{x}^k)-\gamma(c_{\tilde{f}}-\gamma L_f)\|\hat{\mathbf{x}}_{i^k}(\tilde{\mathbf{x}}^k)-\tilde{\mathbf{x}}_{i^k}^k\|_2^2+\frac{L_f}{2}\|\mathbf{x}^k-\tilde{\mathbf{x}}^k\|_2^2\\
		
		\stackrel{\text{(g)}}{=} F(\mathbf{x}^k)-\gamma(c_{\tilde{f}}-\gamma L_f)\|\hat{\mathbf{x}}_{i^k}(\tilde{\mathbf{x}}^k)-\mathbf{x}_{i^k}^k\|_2^2+\frac{L_f}{2}\|\mathbf{x}^k-\tilde{\mathbf{x}}^k\|_2^2,
		\end{array}
		\end{equation}
		where in (a) we used the separability of $g$; (b) follows from the updating rule of the algorithm; in (c) we applied the Descent Lemma \cite{Bertsekas_Book-Parallel-Comp} on $f$; (d) comes from the convexity of $g_i$ and D4; in (e) we used Proposition \ref{Prop_best_response_ncc} and Assumption A3; (f) is due to the Young's inequality, with $\alpha=1$; and (g) comes from D4.\\
		Taking conditional expectations both sides, we have:
		\begin{equation}
		\def\arraystretch{2}\begin{array}{l}
		\mathbb{E}(F(\underline{\mathbf{x}}^{k+1})|\mathcal{F}^{k-1})(\omega)\\
		
		\leq F(\mathbf{x}^k)-\gamma(c_{\tilde{f}}-\gamma L_f)\mathbb{E}(\|\hat{\mathbf{x}}_{\underline{i}^k}(\underline{\tilde{\mathbf{x}}}^k)-\mathbf{x}_{\underline{i}^k}^k\|_2^2|\mathcal{F}^{k-1})(\omega)+\frac{L_f}{2}\mathbb{E}(\|\mathbf{x}^k-\underline{\tilde{\mathbf{x}}}^k\|_2^2|\mathcal{F}^{k-1})(\omega)\\
		
		\stackrel{(a)}{\leq} F(\mathbf{x}^k)-\gamma(c_{\tilde{f}}-\gamma L_f)\sum\limits_{(i,\mathbf{d})\in\mathcal{V}(\omega)}p((i,\mathbf{d})|\omega^{0:k-1})\|\hat{\mathbf{x}}_i(\tilde{\mathbf{x}}^k(\mathbf{d}))-\mathbf{x}_i^k\|_2^2\\+\frac{\delta\psi'\gamma^2L_f}{2}\|\mathbf{w}_\mathbf{x}^k-\mathbf{x}^k\|_2^2\\
		
		\stackrel{(b)}{\leq} F(\mathbf{x}^k)-\Delta\gamma(c_{\tilde{f}}-\gamma L_f)\sum\limits_{i=1}^{N}\|\hat{\mathbf{x}}_i(\tilde{\mathbf{x}}^k(\mathbf{d}_{j_{i,k}}))-\mathbf{x}_i^k\|_2^2+\frac{\delta\psi'\gamma^2L_f}{2}\|\mathbf{w}_\mathbf{x}^k-\mathbf{x}^k\|_2^2\\
		
		= F(\mathbf{x}^k)-\gamma\Delta(c_{\tilde{f}}-\gamma L_f)\|\mathbf{w}_\mathbf{x}^k-\mathbf{x}^k\|_2^2+\frac{\delta\psi'\gamma^2L_f}{2}\|\mathbf{w}_\mathbf{x}^k-\mathbf{x}^k\|_2^2\\
		
		\stackrel{(c)}{=}F(\mathbf{x}^k)-\gamma\left(\Delta(c_{\tilde{f}}-\gamma L_f)-\frac{\gamma\delta\psi'L_f}{2}\right)\|\mathbf{w}_\mathbf{x}^k-\mathbf{x}^k\|_2^2,
		\end{array}
		\label{eq:44}
		\end{equation}
		where (a) follows from \eqref{eq:9gen}; (b) comes from Assumption D3 and holds true for $\gamma\leq\frac{c_{\tilde{f}}}{L_f}$; in (c) we require: $\gamma\leq\frac{2\Delta c_{\tilde{f}}}{2\Delta L_f+\delta\psi'L_f}$.
		Taking expectations both sides of \eqref{eq:44} and rearranging the terms, we get:
		\begin{equation}
		\mathbb{E}(\|\underline{\mathbf{w}}_\mathbf{x}^k-\underline{\mathbf{x}}^k\|_2^2)\leq\frac{2}{\gamma(2\Delta(c_{\tilde{f}}-\gamma L_f)-\gamma\delta\psi'L_f)}\mathbb{E}(F(\underline{\mathbf{x}}^k)-F(\underline{\mathbf{x}}^{k+1})).
		\end{equation}
		Using this result in \eqref{eq:42} we finally have:
		\begin{equation}
		\mathbb{E}(\|M_F(\mathbf{x}^k)\|_2^2)\leq\frac{4(1+(1+L_B+L_E)(1+L_EL_B\delta\psi'\gamma^2))}{p_{\text{min}}\gamma(2\Delta(c_{\tilde{f}}-\gamma L_f)-\gamma\delta\psi'L_f)}\mathbb{E}(F(\underline{\mathbf{x}}^k)-F(\underline{\mathbf{x}}^{k+1})),	
		\end{equation}
		and:
		\begin{equation}
		\def\arraystretch{2}\begin{array}{l}
		K_\epsilon\epsilon
		
		\leq\sum\limits_{k=0}^{K_\epsilon}\mathbb{E}(\|M_F(\mathbf{x}^k)\|_2^2)\\
		
		\leq\sum\limits_{k=0}^{K_\epsilon}\frac{4(1+(1+L_B+L_E)(1+L_EL_B\delta\psi'\gamma^2))}{p_{\text{min}}\gamma(2\Delta(c_{\tilde{f}}-\gamma L_f)-\gamma\delta\psi'L_f)}\mathbb{E}(F(\underline{\mathbf{x}}^k)-F(\underline{\mathbf{x}}^{k+1}))\\
		
		=\frac{4(1+(1+L_B+L_E)(1+L_EL_B\delta\psi'\gamma^2))}{p_{\text{min}}\gamma(2\Delta(c_{\tilde{f}}-\gamma L_f)-\gamma\delta\psi'L_f)}\mathbb{E}(F(\mathbf{x}^0)-F(\underline{\mathbf{x}}^{K_\epsilon+1}))\\
		
		\leq \frac{4(1+(1+L_B+L_E)(1+L_EL_B\delta\psi'\gamma^2))}{p_{\text{min}}\gamma(2\Delta(c_{\tilde{f}}-\gamma L_f)-\gamma\delta\psi'L_f)}(F(\mathbf{x}^0)-F^*).
		\end{array}
		\end{equation}
		This completes the proof.
		
\end{document}